\documentclass[12pt,letterpaper,showkeys]{article} 
\usepackage{graphicx}
\usepackage{amssymb}
\usepackage{amsmath}
\usepackage{mathrsfs}
\graphicspath{{figures}}
\newtheorem{theorem}{Theorem}[section] 
\newtheorem{lemma}[theorem]{Lemma} 
 
\newtheorem{proposition}[theorem]{Proposition} 
\newtheorem{remark}[theorem]{Remark}
\newtheorem{assumption}{Assumption} 

\newcommand{\beq}{\begin{equation}} 
\newcommand{\eeq}{\end{equation}} 
\newcommand{\beqa}{\begin{eqnarray}} 
\newcommand{\eeqa}{\end{eqnarray}} 
\newcommand{\beqas}{\begin{eqnarray*}} 
\newcommand{\eeqas}{\end{eqnarray*}} 
\newcommand{\ba}{\begin{array}} 
\newcommand{\ea}{\end{array}} 
\newcommand{\bi}{\begin{itemize}} 
\newcommand{\ei}{\end{itemize}} 
\newcommand{\gap}{\hspace*{2em}} 
 
\newcommand{\nn}{\nonumber}

\def\eqnok#1{(\ref{#1})}

\def\vgap{\vspace*{.1in}}
\def\QED{\ifhmode\unskip\nobreak\fi\ifmmode\ifinner\else\hskip5pt\fi\fi
 \hbox{\hskip5pt\vrule width5pt height5pt depth1.5pt\hskip1pt}}
\def\AV{{\rm AdjVar }}
\def\balpha{{\bar\alpha}} 
 
\def\bG{{\bar G}}
\def\bx{{\bar x}}
\def\bt{{\ \bullet \ }}

\def\bV{{\bar V}}

\def\cA{{\cal A}}
\def\cD{{\cal D}}

\def\cL{{\cal L}}  
 
\def\cS{{\cal S}} 
\def\cU{{\cal U}} 
\def\dist{{\rm dist}}
\def\Diag{{\rm Diag}}
\def\dbP{{\delta\bar P}}
\def\dD{{\delta D}}
\def\dP{{\delta P}}
\def\dV{{\delta V}}
\def\dY{{\delta Y}}
\def\dZ{{\delta Z}}
\def\eps{{\epsilon}}

\def\hsigma{{\hat \Sigma}}
\def\init{{\rm init}}
\def\lambdam{{\lambda_{\min}}}
\def\lir{{{\underline i}_r}}
\def\uir{{{\bar i}_r}}
\def\lmax{{\lambda_{\max}}}
\def\lmin{{\lambda_{\min}}}
\def\q{{P}}

\def\rank{{\rm rank}}
\def\sign{{\rm sign}}

\def\td{{\tilde d}}
\def\tH{{\tilde H}}

\def\tr{{\rm Tr}}
\def\tx{{\tilde x}}
\def\vrho{{\varrho}}
\def\w{{w}}

\title{An Augmented Lagrangian Approach for Sparse Principal Component Analysis}

\author{
	Zhaosong Lu% 
	\thanks{
	Department of Mathematics, Simon Fraser University, Burnaby, BC, 
	V5A 1S6, Canada. (email: {\tt zhaosong@sfu.ca}). This author was 
    supported in part by NSERC Discovery Grant.} 
	\and
	Yong Zhang 
	\thanks{Department of Mathematics, 
	Simon Fraser University, Burnaby, BC, V5A 1S6, 
    Canada. (email: {\tt yza30@sfu.ca}).}
}

\date{July 10, 2009}

\begin{document}

\maketitle

\begin{abstract}
Principal component analysis (PCA) is a widely used technique for data 
analysis and dimension reduction with numerous applications in science 
and engineering. However, the standard PCA suffers from the fact 
that the principal components (PCs) are usually linear combinations 
of all the original variables, and it is thus often difficult to 
interpret the PCs. To alleviate this drawback, various sparse 
PCA approaches were proposed in literature \cite{Jol95,CaJo95,
JoTrUd03,ZoHaTi06,DaElJoLa07,ShHu07,MoWeAv06,DaBaEl08,JoNeRiSe08}. 
Despite success in achieving sparsity, some important properties 
enjoyed by the standard PCA are lost in these methods such as 
uncorrelation of PCs and orthogonality of loading vectors. Also, 
the total explained variance that they attempt to maximize 
can be too optimistic. In this paper we propose a new formulation 
for sparse PCA, aiming at finding sparse and nearly uncorrelated PCs 
with orthogonal loading vectors while explaining as much of the total 
variance as possible. We also develop a novel augmented 
Lagrangian method for solving a class of nonsmooth constrained 
optimization problems, which is well suited for our formulation of sparse 
PCA. We show that it converges to a {\it feasible} point, and moreover under 
some regularity assumptions, it converges to a stationary point. 
Additionally, we propose two nonmonotone gradient methods for solving 
the augmented Lagrangian subproblems, and establish their global and 
local convergence. Finally, we compare our sparse PCA approach with 
several existing methods on synthetic, random, and real data, respectively. 
The computational results demonstrate that the sparse PCs produced by our approach 
substantially outperform those by other methods in terms of total 
explained variance, correlation of PCs, and orthogonality of loading vectors.
 
\vskip14pt

 
\vskip14pt


\end{abstract}

\section{Introduction} \label{introduction}

Principal component analysis (PCA) is a popular tool for data 
processing and dimension reduction. It has been widely used in 
numerous applications in science and engineering such as biology, 
chemistry, image processing, machine learning and so on. For 
example, PCA has recently been applied to human face recognition, 
handwritten zip code classification and gene expression data 
analysis (see \cite{HaBuBr96,HaTiFr01,AlBrBr00,HaTiEi00}).      

In essence, PCA aims at finding a few linear combinations of the 
original variables, called {\it principal components} (PCs), which 
point in orthogonal directions capturing as much of the variance 
of the variables as possible. It is well known that PCs can be found 
via the eigenvalue decomposition of the covariance matrix $\Sigma$. 
However, $\Sigma$ is typically unknown in practice. Instead, the PCs 
can be approximately computed via the singular value decomposition (SVD) 
of the data matrix or the eigenvalue decomposition of the sample 
covariance matrix. In detail, let $\xi = (\xi^{(1)}, \ldots, \xi^{(p)})$ 
be a $p$-dimensional random vector, and $X$ be an $n \times p$ data 
matrix, which records the $n$ observations of $\xi$. Without loss of 
generality, assume $X$ is centered, that is, the column means of $X$ 
are all $0$. Then the commonly used sample covariance matrix is 
$\hsigma = X^T X/(n-1)$. Suppose the eigenvalue decomposition of 
$\hsigma$ is 
\[
\hsigma = VDV^T.
\]
Then $\eta=\xi V$ gives the PCs, and the columns of $V$ are the  
corresponding loading vectors. It is worth noting that $V$ 
can also be obtained by performing the SVD of $X$ (see, for 
example, \cite{ZoHaTi06}). Clearly, the columns of $V$ are orthonormal 
vectors, and moreover $V^T\hsigma V$ is diagonal. We thus immediately see 
that if $\hsigma=\Sigma$, the corresponding PCs are uncorrelated; 
otherwise, they can be correlated with each other (see Section 
\ref{formulation} for details). We now describe several important 
properties of the PCs obtained by the standard PCA when $\Sigma$ 
is well estimated by $\hsigma$ (see also \cite{ZoHaTi06}): 
\bi
\item[1.] The PCs sequentially capture the maximum variance of the variables 
approximately, thus encouraging minimal information loss as much as possible;
\item[2.] The PCs are nearly uncorrelated, so the explained variance by 
different PCs has small overlap;  
\item[3.] The PCs point in orthogonal directions, that is, their loading 
vectors are orthogonal to each other. 
\ei
In practice, typically the first few PCs are enough to represent the data, 
thus a great dimensionality reduction is achieved. In spite of the popularity 
and success of PCA due to these nice features, PCA has an obvious drawback, 
that is, PCs are usually linear combinations of all $p$ variables and the 
loadings are typically nonzero. This makes it often difficult to interpret 
the PCs, especially when $p$ is large. Indeed, in many applications, the original 
variables have concrete physical meaning. For example in biology, each 
variable might represent the expression level of a gene. In these cases, 
the interpretation of PCs would be facilitated if they were composed only 
from a small number of the original variables, namely, each PC involved a 
small number of nonzero loadings. It is thus imperative to develop sparse 
PCA techniques for finding the PCs with sparse loadings while enjoying the 
above three nice properties as much as possible.       
            
Sparse PCA has been an active research topic for more than a decade. The first 
class of approaches are based on ad-hoc methods by post-processing the PCs 
obtained from the standard PCA mentioned above. For example, Jolliffe \cite{Jol95} 
applied various rotation techniques to the standard PCs for obtaining sparse loading 
vectors. Cadima and Jolliffe \cite{CaJo95} proposed a simple thresholding approach 
by artificially setting to zero the standard PCs' loadings with absolute values 
smaller than a threshold. In recent years, optimization approaches have been 
proposed for finding sparse PCs. They usually formulate sparse PCA into an optimization 
problem, aiming at achieving the sparsity of loadings while maximizing the explained 
variance as much as possible. For instance, Jolliffe et al.\ \cite{JoTrUd03} proposed 
an interesting algorithm, called SCoTLASS, for finding sparse orthogonal loading 
vectors by sequentially maximizing the approximate variance explained by each PC 
under the $l_1$-norm penalty on loading vectors. Zou et al.\ \cite{ZoHaTi06} 
formulated sparse PCA as a regression-type optimization problem and imposed a 
combination of $l_1$- and $l_2$-norm penalties on the regression coefficients. 
d'Aspremont et al.\ \cite{DaElJoLa07} proposed a method, called DSPCA, for 
finding sparse PCs by solving a sequence of semidefinite program relaxations 
of sparse PCA. Shen and Huang \cite{ShHu07} recently developed an approach 
for computing sparse PCs by solving a sequence of rank-one matrix approximation 
problems under several sparsity-inducing penalties. Very recently, Journ\'ee et al.\ 
\cite{JoNeRiSe08} formulated sparse PCA as nonconcave maximization problems with 
$l_0$- or $l_1$-norm sparsity-inducing penalties. They showed that these problems 
can be reduced into maximization of a convex function on a compact set, and they 
also proposed a simple but computationally efficient gradient method for finding 
a stationary point of the latter problems. Additionally, greedy methods were 
investigated for sparse PCA by Moghaddam et al.\ \cite{MoWeAv06} and d'Aspremont 
et al.\ \cite{DaBaEl08}. 

The PCs obtained by the above methods \cite{Jol95,CaJo95,JoTrUd03,ZoHaTi06,DaElJoLa07,
ShHu07,MoWeAv06,DaBaEl08,JoNeRiSe08} are usually sparse. However, the aforementioned 
nice properties of the standard PCs are lost to some extent in these sparse PCs. 
Indeed, the likely correlation among the sparse PCs are not considered in these 
methods. Therefore, their sparse PCs can be quite correlated with each other. Also, 
the total explained variance that these methods attempt to maximize can be too 
optimistic as there may be some overlap among the individual variances of 
sparse PCs. Finally, the loading vectors of the sparse PCs given by these 
methods lack orthogonality except SCoTLASS \cite{JoTrUd03}.                        

In this paper we propose a new formulation for sparse PCA by taking into 
account the three nice properties of the standard PCA, that is, maximal 
total explained variance, uncorrelation of PCs, and orthogonality of loading 
vectors. We also explore the connection of this formulation with the standard 
PCA and show that it can be viewed as a certain perturbation of the standard 
PCA. We further propose a novel augmented Lagrangian method for solving a 
class of nonsmooth constrained optimization problems, which is well suited 
for our formulation of sparse PCA. This method differs from the classical augmented 
Lagrangian method in that: i) the values of the augmented Lagrangian functions 
at their approximate minimizers given by the method are bounded from above; and 
ii) the magnitude of penalty parameters outgrows that of Lagrangian multipliers 
(see Section \ref{aug-method} for details). We show that this method converges to 
a {\it feasible} point, and moreover it converges to a first-order stationary 
point under some regularity assumptions. We also propose two nonmonotone gradient 
methods for minimizing a class of nonsmooth functions over a closed convex set, 
which can be suitably applied to the subproblems arising in our augmented 
Lagrangian method. We further establish global convergence and, under a local 
Lipschitzian error bounds assumption \cite{TseYun09}, local linear rate of convergence 
for these gradient methods. Finally, we compare the sparse PCA approach proposed 
in this paper with several existing methods \cite{ZoHaTi06,DaElJoLa07,ShHu07,
JoNeRiSe08} on synthetic, random, and real data, respectively. The computational 
results demonstrate that the sparse PCs obtained by our approach substantially 
outperform those by the other methods in terms of total explained variance, 
correlation of PCs, and orthogonality of loading vectors.

The rest of paper is organized as follows. 
%In Subsection \ref{notation}, we introduce the notations used in this paper. 
In Section \ref{formulation}, we propose a new formulation for sparse PCA and 
explore the connection of this formulation with the standard PCA. In Section 
\ref{aug-nonsmooth}, we then develop a novel augmented Lagrangian method for 
a class of nonsmooth constrained problems, and propose two nonmonotone gradient 
methods for minimizing a class of nonsmooth functions over a closed convex set.
In Section \ref{aug-spca}, we discuss the applicability and implementation details 
of our augmented Lagrangian method for sparse PCA. The sparse PCA approach proposed 
in this paper is then compared with  several existing methods on synthetic, random, 
and real data in Section \ref{comp}. Finally, we present some concluding 
remarks in Section \ref{concl-remark}. 
 
\subsection{Notation} \label{notation}

In this paper, all vector spaces are assumed to be finite dimensional. The 
symbols $\Re^n$ and $\Re^n_+$ (resp., $\Re^n_-$) denote the $n$-dimensional 
Euclidean space and the nonnegative (resp., nonpositive) orthant of $\Re^n$, 
respectively, and $\Re_{++}$ denotes the set of positive real numbers. The 
space of all $m \times n$ matrices with real entries is denoted by $\Re^{m \times n}$. 
The space of symmetric $n \times n$ matrices is denoted by $\cS^n$. 
Additionally, $\cD^n$ denotes the space of $n \times n$ diagonal matrices. 
For a real matrix $X$, we denote by $|X|$ the absolute value of $X$, that is, 
$|X|_{ij}=|X_{ij}|$ for all $ij$, and by $\sign (X)$ the sign of $X$ whose 
$ij$th entry equals the sign of $X_{ij}$ for all $ij$. Also, the nonnegative 
part of $X$ is denoted by $[X]^+$ whose $ij$th entry is given by $\max\{0,X_{ij}\}$ 
for all $ij$. The rank of $X$ is denoted by $\rank(X)$. Further, 
the identity matrix and the all-ones matrix are denoted 
by $I$ and $E$, respectively, whose dimension should be clear from the context. 
If $X \in \cS^n$ is positive semidefinite, we write $X \succeq 0$. For any $X$, 
$Y \in \cS^n$, we write $X \preceq Y$ to mean $Y-X \succeq 0$. Given 
matrices $X$ and $Y$ in $\Re^{m \times n}$, the standard inner product is defined 
by $X \bt Y := \tr (XY^T)$, where $\tr(\cdot)$ denotes the trace of a matrix, and  
the component-wise product is denoted by $X \odot Y$, whose $ij$th entry is 
$X_{ij}Y_{ij}$ for all $ij$. $\|\cdot\|$ denotes the Euclidean norm and its 
associated operator norm unless it is explicitly stated otherwise. The minimal 
(resp., maximal) eigenvalue of an $n\times n$ symmetric matrix $X$ are denoted by 
$\lmin(X)$ (resp., $\lmax(X)$), respectively, and $\lambda_i(X)$ denotes its 
$i$th largest eigenvalue for $i=1,\ldots,n$. Given a vector $v\in\Re^n$, 
$\Diag(v)$ or $\Diag(v_1,\ldots,v_n)$ denotes a diagonal matrix whose $i$th 
diagonal element is $v_i$ for $i=1,\ldots,n$. Given an $n \times n$ matrix $X$, 
${\widetilde\Diag}(X)$ denotes a diagonal matrix whose $i$th diagonal element 
is $X_{ii}$ for $i=1,\ldots,n$. Let $\cU$ be a real vector space. Given a 
closed convex set $C \subseteq \cU$, let $\dist(\cdot,C): \cU \to \Re_+$ 
denote the distance function to $C$ measured in terms of $\|\cdot\|$, that 
is, 
\beq \label{def_d_ck}
\dist(u,C) := \inf_{\tilde u \in C} \|u - \tilde u\|
\ \ \ \  \forall u \in \cU.
\eeq

\section{Formulation for sparse PCA}
\label{formulation}

In this section we propose a new formulation for sparse PCA by taking 
into account sparsity and orthogonality of loading vectors, and 
uncorrelation of PCs. We also address the connection of our 
formulation with the standard PCA.     

Let $\xi = (\xi^{(1)}, \ldots, \xi^{(p)})$ be a $p$-dimensional random 
vector with covariance matrix $\Sigma$. Suppose $X$ is an $n \times p$ 
data matrix, which records the $n$ observations of $\xi$. Without loss 
of generality, assume the column means of $X$ are $0$. Then the commonly 
used sample covariance matrix of $\xi$ is $\hsigma = X^T X/(n-1)$. For 
any $r$ loading vectors represented as $V = [V_1, \ldots, V_r] \in \Re^{p 
\times r}$ where $1 \le r \le p$, the corresponding components are given 
by $\eta=(\eta^{(1)}, \ldots, \eta^{(r)})=\xi V$, which are linear 
combinations of $\xi^{(1)}, \ldots, \xi^{(p)}$. Clearly, the covariance 
matrix of $\eta$ is $V^T\Sigma V$, and thus the components $\eta^{(i)}$ and 
$\eta^{(j)}$ are uncorrelated if and only if the $ij$th entry of 
$V^T\Sigma V$ is zero. Also, the total explained variance by the components 
$\eta^{(i)}$'s equals, if they are uncorrelated, the sum of the 
individual variances of $\eta^{(i)}$'s, that is, 
\[
\sum^{r}_{i=1} V^T_i \Sigma V_i = \tr(V^T \Sigma V). 
\]       
Recall that our aim is to find a set of sparse and orthogonal loading vectors 
$V$ so that the corresponding components $\eta^{(1)}, \ldots, \eta^{(r)}$ are 
uncorrelated and explain as much variance of the original variables $\xi^{(1)}, 
\ldots, \xi^{(p)}$ as possible. It appears that our goal can be achieved 
by solving the following problem:
\beq \label{diag-form}
\ba{rl}
\max\limits_{V \in \Re^{n \times r}} & \tr(V^T \Sigma V) - \rho \bt |V| \\ [4pt]
\mbox{s.t.} & V^T \Sigma V \ \mbox{is \ diagonal}, \\ [5pt]
& V^T V = I,   
\ea
\eeq 
where $\rho\in\Re^{p \times r}_+$ is a tunning parameter for controlling 
the sparsity of $V$. However, the covariance matrix $\Sigma$ is typically 
unknown and can only be approximated by the sample covariance matrix 
$\hsigma$. It looks plausible to modify \eqnok{diag-form} by simply 
replacing $\Sigma$ with $\hsigma$ at a glance. Nevertheless, such a 
modification would eliminate all optimal solutions $V^*$ of \eqnok{diag-form} 
from consideration since $(V^*)^T \hsigma V^*$ is generally non-diagonal. 
For this reason, given a sample covariance $\hsigma$, we consider the 
following formulation for sparse PCA, which can be viewed as a modification 
of problem \eqnok{diag-form},
%but nearly diagonal provided that the sample size $n$ is sufficiently large. 
%In the latter case, 
%the explained variance $(V^*)^T \Sigma V^*$ can in turn be estimated by 
%$(V^*)^T \hsigma V^*$. Given a sample covariance $\hsigma$, we thus see 
%that \eqnok{diag-form} can be approximately solved by the following problem:
\beq \label{diag-approx}
\ba{rl}
\max\limits_{V \in \Re^{n \times r}} & \tr(V^T \hsigma V) - \rho \bt |V| \\ [4pt]
\mbox{s.t.} &   |V^T_i \hsigma V_j| \le \Delta_{ij}  \  \  \
\forall i \neq j, \\ [5pt]
& V^T V = I,   
\ea
\eeq 
where $\Delta_{ij} \ge 0$ $(i \neq j)$  are the parameters for controlling 
the correlation of the components corresponding to $V$. Clearly, $\Delta_{ij} 
= \Delta_{ji}$ for all $i \neq j$.     
        
We next explore the connection of formulation \eqnok{diag-approx} with the 
standard PCA. Before proceeding, we state a technical lemma as follows 
that will be used subsequently. Its proof can be found in \cite{OveWom93}.       

\begin{lemma} \label{sum-eig1} 
Given any $\hsigma \in \cS^n$ and integer $1\le r \le n$, define 
\beq \label{ir}
\lir = \max\{1 \le i \le n: 
\lambda_i(\hsigma) > \lambda_r(\hsigma)\}, \gap  \uir = \max\{1 \le i \le n: 
\lambda_i(\hsigma) = \lambda_r(\hsigma)\},
\eeq
and let $f^*$ be the optimal value of   
\beq \label{relax-prob0}
\max \{\tr(\hsigma Y): \ 0 \preceq Y \preceq I, \ \tr(Y) = r \}.   
\eeq 
Then, $f^* = \sum^r_{i=1} \lambda_i(\hsigma)$, and $Y^*$ is an optimal solution 
of \eqnok{relax-prob0} if and only if $Y^*=U^*_1 U^{*T}_1 + U^*_2 P^* U^{*T}_2 $, 
where $P^* \in \cS^{\uir-\lir}$ satisfies $0 \preceq P^* \preceq I$ and $\tr(P^*) 
= r - \lir$, and  $U^*_1 \in \Re^{n \times \lir}$ and  $U^*_2 \in  \Re^{n \times (\uir-\lir)}$ 
are the matrices whose columns consist of the orthonormal eigenvectors of $\hsigma$ 
corresponding to the eigenvalues $(\lambda_1(\hsigma), \ldots, \lambda_\lir(\hsigma))$ 
and  $(\lambda_{\lir+1}(\hsigma), \ldots, \lambda_\uir(\hsigma))$, respectively.  
\end{lemma}

We next address the relation between the eigenvectors of $\hsigma$ and the 
solutions of problem \eqnok{diag-approx} when $\rho=0$ and $\Delta_{ij}=0$ 
for all $i \neq j$.

\begin{proposition} \label{prop1}
Suppose for problem \eqnok{diag-approx} that $\rho=0$ and $\Delta_{ij}=0$ for all 
$i \neq j$. Let $f^*$ be the optimal value of \eqnok{diag-approx}. Then, $f^* = 
\sum^r_{i=1} \lambda_i(\hsigma)$, and $V^*\in \Re^{n \times r}$ is an optimal 
solution of \eqnok{diag-approx} if and only if the columns of $V^*$ consist of 
the orthonormal eigenvectors of $\hsigma$ corresponding to $r$ largest 
eigenvalues of $\hsigma$.
\end{proposition}

\begin{proof}
We first show that $f^*=\sum^r_{i=1} \lambda_i(\hsigma)$. Indeed, let $U$ be an 
$n\times r$ matrix whose columns consist of the orthonormal eigenvectors of 
$\hsigma$ corresponding to $r$ largest eigenvalues of $\hsigma$. We then see that 
$U$ is a feasible solution of \eqnok{diag-approx} and $\tr(U^T \hsigma U) = \sum^r_{i=1} 
\lambda_i(\hsigma)$. It follows that $f^* \ge \sum^r_{i=1} \lambda_i(\hsigma)$. On the 
other hand, we observe that $f^*$ is bounded above by the optimal value of 
\[
\max \{\tr(V^T \hsigma V): \ V^T V = I, \ V\in\Re^{n \times r}\}.   
\]
We know from \cite{Fan49} that its optimal value equals $\sum^r_{i=1} \lambda_i(\hsigma)$. 
Therefore, $f^*=\sum^r_{i=1} \lambda_i(\hsigma)$ holds and $U$ is an optimal 
solution of \eqnok{diag-approx}. It also implies that the ``if'' part of this 
proposition holds. We next show that the ``only if'' part also holds. Let 
$V^*\in \Re^{n \times r}$ be an optimal solution of \eqnok{diag-approx}, and define 
$Y^*=V^*V^{*T}$. Then, we have $V^{*T}V^*=I$, which yields $0 \preceq Y^* \preceq I$ 
and $\tr(Y^*)=r$. Hence, $Y^*$ is a feasible solution of \eqnok{relax-prob0}. 
Using the fact that $f^*=\sum^r_{i=1} \lambda_i(\hsigma)$, we then have
\[ 
\tr(\hsigma Y^*)=\tr(V^{*T}\hsigma V^*)=\sum^r_{i=1} \lambda_i(\hsigma),
\]
which together with Lemma \ref{sum-eig1} implies that $Y^*$ is an optimal solution 
of \eqnok{relax-prob0}. Let $\lir$ and $\uir$ be defined in \eqnok{ir}. Then, it follows 
from Lemma \ref{sum-eig1} that $Y^*=U^*_1 U^{*T}_1 + U^*_2 P^* U^{*T}_2$, 
where $P^* \in \cS^{\uir-\lir}$ satisfies $0 \preceq P^* \preceq I$ and $\tr(P^*) 
= r - \lir$, and  $U^*_1 \in \Re^{n \times \lir}$ and  $U^*_2 \in  \Re^{n \times (\uir-\lir)}$ 
are the matrices whose columns consist of the orthonormal eigenvectors of $\hsigma$ 
corresponding to the eigenvalues $(\lambda_1(\hsigma), \ldots, \lambda_\lir(\hsigma))$ 
and  $(\lambda_{\lir+1}(\hsigma), \ldots, \lambda_\uir(\hsigma))$, respectively. 
Thus, we have 
\beq \label{U-prop}
\hsigma U^*_1 = U^*_1 \Lambda, \ \ \ \hsigma U^*_2 = \lambda_r(\hsigma)U^*_2,   
\eeq
where $\Lambda = \Diag(\lambda_1(\hsigma), \ldots, \lambda_\lir(\hsigma))$. In 
addition, it is easy to show that $\rank(Y^*) = \lir + \rank(P^*)$. Since $Y^*=V^*V^{*T}$ 
and $V^{*T}V^*=I$, we can observe that $\rank(Y^*)=r$. Hence, $\rank(P^*)=r-\lir$, 
which implies that $P^*$ has only $r-\lir$ nonzero eigenvalues. Using this fact 
and the relations $0 \preceq P^* \preceq I$ and $\tr(P^*) = r - \lir$, we can further 
conclude that $r-\lir$ eigenvalues of $P^*$ are $1$ and the rest are $0$. Therefore, 
there exists $W\in \Re^{(\uir-\lir) \times (r-\lir)}$ such that 
\beq \label{W}
W^TW = I, \ \ \ \ P^* = WW^T.
\eeq    
It together with $Y^*=U^*_1 U^{*T}_1 + U^*_2 P^* U^{*T}_2$ implies that 
$Y^* = U^*U^{*T}$, where $U^* = [U^*_1 \ \ U^*_2 W]$. In view of \eqnok{W} and 
the identities $U^{*T}_1 U^*_1=I$, $U^{*T}_2 U^*_2=I$ and $U^{*T}_1 U^*_2=0$, we 
see that $U^{*T}U^*=I$. Using this result, and the relations $V^{*T}V^*=I$ and 
$Y^*=U^* U^{*T}=V^* V^{*T}$, it is not hard to see that the columns of $U^*$ and 
$V^*$ form an orthonormal basis for the range space of $Y^*$, respectively. Thus, 
$V^*=U^*Q$ for some $Q\in\Re^{r \times r}$ satisfying $Q^TQ=I$. Now, 
let $D=V^{*T}\hsigma V^*$. By the definition of $V^*$, we know that $D$ is 
an $r \times r$ diagonal matrix. Moreover, in view of \eqnok{U-prop}, \eqnok{W}, the 
definition of $U^*$, and the relations $V^*=U^*Q$, $U^{*T}_1 U^*_1=I$, 
$U^{*T}_2 U^*_2=I$ and $U^{*T}_1 U^*_2=0$, we have
\beqa 
D &=& V^{*T} \hsigma V^* \ = \ Q^T U^{*T} \hsigma U^* Q \ = \ Q^T \left[\ba{c}
U^{*T}_1 \\
W^T U^{*T}_2
\ea\right] \hsigma \left[ U^*_1 \ \ U^*_2W\right] Q \nn \\
&=& Q^T \left[\ba{cc} 
\Lambda & 0 \\
0 & \lambda_r(\hsigma) I
\ea\right] Q, \label{D}
\eeqa
which together with $Q^TQ=I$ implies that $D$ is similar to the diagonal matrix 
appearing on the right-hand side of \eqnok{D}.  
%Further, since 
%$D$ and $\Lambda$ both are diagonal matrices, they share identical diagonal entries upon 
%some permutations if necessary. 
Hence, the diagonal elements of $D$ consist of $r$ largest eigenvalues of $\hsigma$. 
In addition, let $Q_1\in\Re^{\lir \times r}$ and $Q_2\in\Re^{(r-\lir) \times r}$ be 
the submatrices corresponding to the first $\lir$ and the last $r-\lir$ rows of $Q$, 
respectively. Then, in view of the definition of $U^*$ and $V^*=U^*Q$, we have 
\[
[U^*_1 \ \ U^*_2 W] \ =  \ U^* = V^* Q^T  \ = \ [V^*Q^T_1 \ \ V^*Q^T_2]. 
\]
Thus, we obtain that $U^*_1 = V^*Q^T_1$ and $U^*_2 W=V^*Q^T_2$. Using these 
identities, \eqnok{U-prop}, \eqnok{D}, and the relation $V^* = U^* Q$,  we have 
\beqas
\hsigma V^* &=& \hsigma U^* Q = \hsigma \left[ U^*_1 \ \ U^*_2W\right] Q 
\ = \ [U^*_1 \Lambda \ \  \lambda_r(\hsigma)U^*_2 W] Q \\
&=&  [V^*Q^T_1\Lambda \ \ \lambda_r(\hsigma) V^*Q^T_2] Q \ = \ V^* Q^T \left[\ba{cc} 
\Lambda & 0 \\
0 & \lambda_r(\hsigma) I
\ea\right] Q \ = \ V^* D.
\eeqas   
It follows that the columns of $V^*$ consist of the orthonormal eigenvectors of $\hsigma$ 
corresponding to $r$ largest eigenvalues of $\hsigma$, and thus the ``only if'' part of 
this proposition holds.     
\end{proof}

\gap
 
 From the above proposition, we see that when $\rho=0$ and $\Delta_{ij}=0$ for all $i \neq j$, 
each solution of \eqnok{diag-approx} consists of the orthonormal eigenvectors of $\hsigma$ 
corresponding to $r$ largest eigenvalues of $\hsigma$, which can be computed from the 
eigenvalue decomposition of $\hsigma$. Therefore, the loading vectors obtained from 
\eqnok{diag-approx} are the same as those given by the standard PCA when applied to 
$\hsigma$. On the other hand, when $\rho$ and $\Delta_{ij}$ for all $i \neq j$ are 
small, the loading vectors found by \eqnok{diag-approx} can be viewed as an 
approximation to the ones provided by the standard PCA. We will propose suitable 
methods for solving \eqnok{diag-approx} in Sections \ref{aug-nonsmooth} and 
\ref{aug-spca}.          

\section{Augmented Lagrangian method for nonsmooth constrained 
nonlinear programming} \label{aug-nonsmooth}

In this section we propose a novel augmented Lagrangian method for a class of 
nonsmooth constrained nonlinear programming problems, which is well suited for 
formulation \eqnok{diag-approx} of sparse PCA. In particular, we study first-order 
optimality conditions in Subsection \ref{1st-optcond}. In Subsection \ref{aug-method}, 
we develop an augmented Lagrangian method and establish its global convergence. In 
Subsection \ref{grad-method}, we propose two nonmonotone gradient methods for minimizing 
a class of nonsmooth functions over a closed convex set, which can be suitably applied to 
the subproblems arising in our augmented Lagrangian method. We also establish global and 
local convergence for these gradient methods.  

\subsection{First-order optimality conditions}
\label{1st-optcond}

In this subsection we introduce a class of nonsmooth constrained nonlinear 
programming problems and study first-order optimality conditions for them. 

Consider the nonlinear programming problem
\beq \label{nlp}
\begin{array}{rl}
\min & f(x) + \q(x) \\ 
\mbox{s.t.} & g_i(x) \le 0, \ \ i = 1, \ldots, m, \\
& h_i(x) = 0, \ \ i = 1, \ldots, p, \\
& x \in X. 
\end{array}   
\eeq 
We assume that the functions $f: \Re^n \to \Re$, $g_i: \Re^n \to \Re$, 
$i=1, \ldots, m$, and $h_i: \Re^n \to \Re$, $i=1, \ldots, p$, are 
continuously differentiable, and that the function $\q: \Re^n \to \Re$ 
is convex but not necessarily smooth, and that the set $X \subseteq \Re^n$ 
is closed and convex. For convenience of the subsequent presentation, 
we denote by $\Omega$ the feasible region of problem \eqnok{nlp}.  

Before establishing first-order optimality conditions for problem  
\eqnok{nlp}, we describe a general constraint qualification condition 
for \eqnok{nlp}, that is, Robinson's condition that was proposed in 
\cite{Rob76}.    

Let $x\in\Re^n$ be a feasible point of problem \eqnok{nlp}. We denote  
the set of active inequality constraints at $x$ as   
\[
\cA(x) = \{1\le i \le m: g_i(x) = 0 \}. 
\]
In addition, $x$ is said to satisfy {\it Robinson's condition} if 
\beq \label{rob-cond}
\left \{\left[\ba{c} g'(x)d - v \\ 
h'(x) d 
\ea \right]: d \in T_{X}(x), v \in \Re^m, v_i \le 0, 
i \in \cA(x)\right\} = \Re^m \times \Re^p,
\eeq
where $g'(x)$ and $h'(x)$ denote the Jacobian of the functions $g=(g_1,
\ldots,g_m)$ and $h=(h_1,\ldots,h_p)$ at $x$, respectively. Other 
equivalent expressions of Robinson's condition can be found, for 
example, in \cite{Rob76,Rob83,Rus06}. 

The following proposition demonstrates that Robinson's condition is indeed  
a constraint qualification condition for problem \eqnok{nlp}. For the sake of 
completeness, we include a brief proof for it. 

\begin{proposition} \label{constrq}
Given a feasible point $x\in\Re^n$ of problem \eqnok{nlp}, let $T_{\Omega}(x)$ 
be the tangent cone to $\Omega$ at $x$, and $(T_{\Omega}(x))^\circ$ be its polar 
cone. If Robinson's condition \eqnok{rob-cond} holds at $x$, then
\beqa
T_{\Omega}(x) &=& \left\{d\in T_X(x): \ba{ll} 
d^T \nabla g_i(x)  \le 0, & i \in \cA(x), \\
d^T \nabla h_i(x) = 0, & i=1, \ldots, p \ea
\right\},  \nn \\
(T_{\Omega}(x))^\circ &=& \left\{\sum\limits_{i\in \cA(x)}  \lambda_i 
\nabla g_i(x) +  \sum\limits_{i=1}^p \mu_i \nabla h_i(x) + N_X(x): 
\lambda \in \Re^m_+, \ \mu\in\Re^p\right\}, \label{tangent-p}
\eeqa 
where $T_X(x)$ and $N_X(x)$ are the tangent and normal cones to $X$ at $x$, 
respectively.
\end{proposition}

\begin{proof}
By Theorem A.10 of \cite{Rus06}, we see that Robinson's condition \eqnok{rob-cond} 
implies that the assumption of Theorem $3.15$ of \cite{Rus06} is satisfied with 
\[
x_0 = x, \ \ \ X_0 = X, \ \ \ Y_0 = \Re^m_- \times \Re^p, \ \ \ 
g(\cdot) = (g_1(\cdot); \ldots ; g_m(\cdot); h_1(\cdot); \ldots ; h_p(\cdot)).  
\]
The first statement then follows from Theorem $3.15$ of \cite{Rus06} with 
the above $x_0$, $X_0$, $Y_0$ and $g(\cdot)$. Further, let $A(x)$ denote the 
matrix whose rows are the gradients of all active constraints at $x$ in the 
same order as they appear in \eqnok{nlp}. Then, Robinson's condition \eqnok{rob-cond} 
implies that the assumptions of Theorem $2.36$ of \cite{Rus06} are satisfied 
with 
\[
A = A(x), \ \ \ K_1 = T_X(x), \ \ \ K_2 = \Re^{|\cA(x)|}_- \times \Re^p. 
\]
Let $K = \{d\in K_1: A d \in K_2\}$. Then, it follows from Theorem $2.36$ of  
\cite{Rus06} that 
\[
(T_{\Omega}(x))^\circ = K^\circ =  K_1^\circ + \{A^T \xi: \xi \in K_2^\circ\},
\]
which together with the identity $(T_X(x))^\circ = N_X(x)$ and the 
definitions of $A$, $K_1$ and $K_2$, implies that the second statement 
holds. 
\end{proof}

\vgap

We are now ready to establish first-order optimality conditions for problem 
\eqnok{nlp}.

\begin{theorem} \label{opt-thm}
Let $x^*\in\Re^n$ be a local minimizer of problem \eqnok{nlp}. Assume that 
Robinson's condition \eqnok{rob-cond} is satisfied at $x^*$. Then there exist 
Lagrange multipliers $\lambda \in \Re^m_+$ and $\mu \in \Re^p$ such that
\beq \label{1st-opt}
0 \in \nabla f(x^*) + \partial \q(x^*) + \sum\limits_{i=1}^m \lambda_i 
\nabla g_i(x^*) +  \sum\limits_{i=1}^p \mu_i \nabla h_i(x^*) + N_X(x^*),
\eeq
and
\beq \label{complement}
\lambda_i g_i(x^*) = 0, \  \ i=1,\ldots,m.
\eeq
Moreover, the set of Lagrange multipliers $(\lambda, \mu)\in\Re^m_+ 
\times \Re^p$ satisfying the above conditions, denoted by $\Lambda(x^*)$, 
is convex and compact. 
\end{theorem}  

\begin{proof} 
We first show that 
\beq \label{basic-opt}
d^T\nabla f(x^*) + \q'(x^*; d) \ge 0 \ \ \ \forall d \in  T_{\Omega}(x^*).
\eeq
%Let $F(x)$ denote the objective function of problem \eqnok{nlp}.
Let $d \in  T_{\Omega}(x^*)$ be arbitrarily chosen. Then, there exist sequences 
$\{x^k\}^{\infty}_{k=1} \subseteq \Omega$ and $\{t_k\}^{\infty}_{k=1}\subseteq 
\Re_{++}$ such that $t_k \downarrow 0$ and 
\[
d = \lim\limits_{k\to \infty} \frac{x^k-x^*}{t_k}.
\] 
Thus, we have $x^k = x^* + t_k d + o(t_k)$. Using this relation along with the 
fact that the function $f$ is differentiable and $\q$ is convex in $\Re^n$, we 
can have  
\beq \label{estimate}
f(x^* + t_k d) - f(x^k) = o(t_k), \ \ \ \q(x^* + t_k d) - \q(x^k) = o(t_k),
\eeq
where the first equality follows from the Mean Value Theorem while the second 
one comes from Theorem $10.4$ of \cite{Rock70}. Clearly, $x^k \to x^* $. 
This together with the assumption that $x^*$ is a local minimizer of \eqnok{nlp}, 
implies that 
\beq \label{local-min}
f(x^k) + \q(x^k)  \ge f(x^*) + \q(x^*)
\eeq
when $k$ is sufficiently large. In view of \eqnok{estimate} and \eqnok{local-min}, 
we obtain that 
\beqas
d^T\nabla f(x^*) + \q'(x^*; d) &=& \lim\limits_{k \to \infty} \frac{f(x^* + t_k d) 
- f(x^*)}{t_k} + \lim\limits_{k \to \infty} \frac{\q(x^* + t_k d) - \q(x^*)}{t_k}, \\ [5pt] 
&=& \lim\limits_{k \to \infty} \left[\frac{f(x^k)+\q(x^k)-f(x^*)-\q(x^*)}{t_k} + 
\frac{o(t_k)}{t_k} \right], \\ [6pt]
&=& \lim\limits_{k \to \infty} \frac{f(x^k)+\q(x^k)-f(x^*)-\q(x^*)}{t_k} \ \ge \ 0,
\eeqas 
and hence \eqnok{basic-opt} holds. 

For simplicity of notations, let $T_{\Omega}^\circ = (T_{\Omega}(x^*))^\circ$ and 
$S=-\nabla f(x^*) - \partial \q(x^*)$. We next show that $S \cap T_{\Omega}^\circ 
\neq \emptyset$. Suppose for contradiction that $S \cap T_{\Omega}^\circ = \emptyset$. 
This together with the fact that $S$ and $T_{\Omega}^\circ$ are nonempty closed convex 
sets and $S$ is bounded, implies that there exists some $d \in\Re^n$ such that 
$d^T y \le 0$ for any $y \in T_{\Omega}^\circ$, and $d^T y \ge 1$ for any $y \in S$. 
Clearly, we see that $d\in (T_{\Omega}^\circ)^\circ = T_{\Omega}(x^*)$, and  
\[
1 \le \inf_{y\in S} d^T y = \inf_{z \in \partial \q(x^*)} d^T(-\nabla f(x^*)-z) 
= -d^T \nabla f(x^*) - \sup_{z \in \partial \q(x^*)} d^T z = -d^T \nabla f(x^*) 
- \q'(x^*; d),    
\]
which contradicts \eqnok{basic-opt}. Hence, we have $S \cap T_{\Omega}^\circ 
\neq \emptyset$. Using this relation, \eqnok{tangent-p}, the definitions of $S$ and 
$\cA(x^*)$,  and letting $\lambda_i=0$ for $i\notin \cA(x^*)$, we easily see that 
\eqnok{1st-opt} and \eqnok{complement} hold.  

In view of the fact that $\partial \q(x^*)$ and $N_X(x^*)$ are closed and convex, and 
moreover $\partial \q(x^*)$ is bounded, we know that $\partial \q(x^*)+ N_X(x^*)$ 
is closed and convex. Using this result, it is straightforward to see that $\Lambda(x^*)$ 
is closed and convex. We next show that $\Lambda(x^*)$ is bounded. Suppose for 
contradiction that $\Lambda(x^*)$ is unbounded. Then, there exists a sequence 
$\{(\lambda^k, \mu^k)\}^{\infty}_{k=1} \subseteq \Lambda(x^*)$ such that 
$\|(\lambda^k, \mu^k)\| \to \infty$, and 
\beq \label{1st-opt-seq}
0 = \nabla f(x^*) + z^k + \sum\limits_{i=1}^m \lambda^k_i 
\nabla g_i(x^*) +  \sum\limits_{i=1}^p \mu^k_i \nabla h_i(x^*) + v^k 
\eeq         
for some $\{z^k\}^{\infty}_{k=1} \subseteq \partial \q(x^*)$ and 
$\{v^k\}^{\infty}_{k=1} \subseteq N_X(x^*)$. Let 
$(\bar \lambda^k, \bar \mu^k) = (\lambda^k,\mu^k) /{\|(\lambda^k, \mu^k)\|}$. 

By passing to a subsequence if necessary, we can assume that $(\bar \lambda^k, 
\bar \mu^k) \to (\bar\lambda, \bar\mu)$. We clearly see that $\|(\bar\lambda,
\bar\mu)\|=1$, $\bar\lambda\in\Re^m_+$, and $\bar\lambda_i=0$ for $i\notin\cA(x^*)$. 
Note that $\partial \q(x^*)$ is bounded and $N_X(x^*)$ is a closed cone. In view of 
this fact, and upon dividing both sides of \eqnok{1st-opt-seq} by $\|(\lambda^k, \mu^k)\|$ 
and taking limits on a subsequence if necessary, we obtain that             
\beq \label{grad-eqn}
0 = \sum\limits_{i=1}^m \bar\lambda_i \nabla g_i(x^*) +  \sum\limits_{i=1}^p \bar\mu_i 
\nabla h_i(x^*) + \bar v 
\eeq     
for some $\bar v\in N_X(x^*)$. Since Robinson's condition \eqnok{rob-cond} 
is satisfied at $x^*$, there exist $d\in T_X(x^*)$ and $v \in \Re^m$ such that $v_i \le 0$ 
for $i\in\cA(x^*)$, and 
\[
\ba{rcl}
d^T \nabla g_i(x^*) - v_i &=& -\bar\lambda_i \ \ \ \forall i \in \cA(x^*), \\ [4pt]
d^T \nabla h_i(x^*) &=& - \bar\mu_i, \ \ \ i = 1, \ldots, p.
\ea
\]  
Using these relations, \eqnok{grad-eqn} and the fact that $d\in T_X(x^*)$, 
$\bar v\in N_X(x^*)$, $\bar\lambda\in\Re^m_+$, and $\bar\lambda_i=0$ for $i\notin \cA(x^*)$, 
we have  
\[
\ba{lcl}
\sum\limits_{i=1}^m  \bar\lambda^2_i + \sum\limits_{i=1}^p \bar\mu^2_i & \le & 
- \sum\limits_{i=1}^m \bar\lambda_i d^T\nabla g_i(x^*) - \sum\limits_{i=1}^p \bar\mu_i 
d^T \nabla h_i(x^*), \\ 
& = & - d^T\left(\sum\limits_{i=1}^m \bar\lambda_i \nabla g_i(x^*) + \sum\limits_{i=1}^p 
\bar\mu_i \nabla h_i(x^*)\right) \ = \ d^T \bar v \ \le \ 0. 
\ea
\]
It yields $(\bar\lambda,\bar\mu)=(0,0)$, which contradicts the identity 
$\|(\bar\lambda,\bar\mu)\|=1$. Thus, $\Lambda(x^*)$ is bounded. 
\end{proof}

\subsection{Augmented Lagrangian method for \eqnok{nlp}}
\label{aug-method}

%In this subsection we propose a novel augmented Lagrangian method for problem 
%\eqnok{nlp} and establish its global convergence. 
For a convex program, it is known that under some mild assumptions, any accumulation 
point of the sequence generated by the classical augmented Lagrangian method is an 
optimal solution  (e.g., see \cite{}). Nevertheless, when problem \eqnok{nlp} is 
a nonconvex program, especially when the function $h_i$ is not affine or $g_i$ is 
nonconvex, the classical augmented Lagrangian method may not even converge to a 
feasible point. To alleviate this drawback, we propose a novel augmented Lagrangian 
method for problem \eqnok{nlp} and establish its global convergence in this 
subsection.

Throughout this subsection, we make the following assumption for 
problem \eqnok{nlp}.

\begin{assumption} \label{assum-aug}
Problem \eqnok{nlp} is feasible, and moreover at least a feasible 
solution, denoted by  $x^{\rm feas}$, is known.  
\end{assumption}

It is well-known that for problem \eqnok{nlp} the associated augmented Lagrangian 
function $L_{\vrho}(x,\lambda,\mu): \Re^n \times \Re^m \times \Re^p \to \Re$ is 
given by 
\beq \label{aug} 
L_{\vrho}(x,\lambda,\mu) := \w(x) + \q(x), 
\eeq
where 
\beq \label{wx}
\w(x) := f(x) + \frac{1}{2\vrho}(\|[\lambda+\vrho g(x)]^+\|^2-\|\lambda\|^2) + 
\mu^T h(x) + \frac{\vrho}{2} \|h(x)\|^2,
\eeq
and $\vrho>0$ is a penalty parameter (e.g., see \cite{Bert99,Rus06}). 
Roughly speaking, an augmented Lagrangian method, when applied to problem \eqnok{nlp}, 
solves a sequence of subproblems in the form of 
\[
\min\limits_{x\in X} L_{\vrho}(x,\lambda,\mu)
\]  
while updating the Lagrangian multipliers $(\lambda, \mu)$ and the penalty parameter 
$\vrho$. 

Let $x^{\rm feas}$ be a known feasible point of \eqnok{nlp} (see Assumption \ref{assum-aug}). 
We now describe the algorithm framework of a novel augmented Lagrangian method as follows.

\gap
 
\noindent
%\begin{minipage}[h]{6.6 in}
{\bf Algorithm framework of augmented Lagrangian method:} \\ [5pt]
Let $\{\epsilon_k\}$ be a positive decreasing sequence. Let $\lambda^0\in\Re^m_+$, 
$\mu^0\in\Re^p$, $\vrho_0 >0$, $\tau>0$, $\sigma > 1$ be given. Choose an arbitrary 
initial point $x^0_\init\in X$ and constant $\Upsilon \ge \max\{f(x^{\rm feas}), L_{\vrho_0}
(x^0_\init,\lambda^0,\mu^0)\}$. Set $k=0$. 
\begin{itemize}
\item[1)]
Find an approximate solution $x^k\in X$ for the subproblem 
\beq \label{inner-prob}
\min\limits_{x\in X} L_{\vrho_k}(x,\lambda^k,\mu^k) 
\eeq
such that 
\beq \label{inner-cond}
\dist\left(-\nabla \w(x^k), \partial \q(x^k) + N_X(x^k)\right) \le \epsilon_k, 
\ \ \ \ L_{\vrho_k}(x^k, \lambda^k, \mu^k) \le \Upsilon. 
\eeq
\item[2)]
Update Lagrange multipliers according to
\beq \label{lag-update}
\lambda^{k+1} := [\lambda^k + \vrho_k g(x^k)]^+, \ \ \ \
\mu^{k+1} := \mu^k + \vrho_k h(x^k). 
\eeq
\item[3)]
Set $\vrho_{k+1} := \max\left\{\sigma\vrho_k, \|\lambda^{k+1}\|^{1+\tau}, 
\|\mu^{k+1}\|^{1+\tau}\right\}$. 
\item[4)]
Set $k \leftarrow k+1$ and go to step 1). 
\end{itemize}
\noindent
{\bf end}
%\end{minipage}

\vgap 

The above augmented Lagrangian method differs from the classical augmented 
Lagrangian method in that: i) the values of the augmented Lagrangian functions 
at their approximate minimizers given by the method are bounded from above 
(see Step 1)); and ii) the magnitude of penalty parameters outgrows that of 
Lagrangian multipliers (see Step 2)).

In addition, to make the above augmented Lagrangian method complete, we 
need address how to find an approximate solution $x^k\in X$ for subproblem 
\eqnok{inner-prob} satisfying \eqnok{inner-cond} as required in Step 1). We 
will leave this discussion to the end of this subsection. For the time being, 
we establish the main convergence result regarding this method for solving 
problem \eqnok{nlp}. 

\begin{theorem} \label{main-thm}
Assume that $\epsilon_k \to 0$. Let $\{x^k\}$ be the sequence generated by 
the above augmented Lagrangian method satisfying \eqnok{inner-cond}. Suppose 
that a subsequence $\{x^k\}_{k \in K}$ converges to $x^*$. Then, the following 
statements hold:
\bi
\item[(a)] 
$x^*$ is a feasible point of problem \eqnok{nlp};
\item[(b)]
Further, if Robinson's condition \eqnok{rob-cond} is satisfied at $x^*$, then 
the subsequence $\{(\lambda^{k+1},\mu^{k+1})\}_{k\in K}$ is bounded, and 
each accumulation point $(\lambda^*,\mu^*)$ of $\{(\lambda^{k+1},\mu^{k+1})\}_{k \in K}$  
is the vector of Lagrange multipliers satisfying the first-order optimality 
conditions \eqnok{1st-opt}-\eqnok{complement} at $x^*$.  
\ei
\end{theorem} 

\begin{proof} 
In view of \eqnok{aug}, \eqnok{wx} and the second relation in \eqnok{inner-cond}, 
we have 
\[
 f(x^k) + \q(x^k) + \frac{1}{2\vrho_k}(\|[\lambda^k+\vrho_k g(x^k)]^+\|^2-\|\lambda^k\|^2) 
+ (\mu^k)^T h(x^k) + \frac{\vrho_k}{2} \|h(x^k)\|^2 \le \Upsilon \ \ \ \forall k.
\] 
It follows that 
\[
\|[\lambda^k/\vrho_k+g(x^k)]^+\|^2 +  \|h(x^k)\|^2 \le 
2[\Upsilon -f(x^k)-g(x^k)-(\mu^k)^T h(x^k)]/\vrho_k +(\|\lambda_k\|/\vrho_k)^2.
\]
Noticing that $\vrho_0 >0$ $\tau>0$, and $\vrho_{k+1} = \max\left\{\sigma\vrho_k, 
\|\lambda^{k+1}\|^{1+\tau}, \|\mu^{k+1}\|^{1+\tau}\right\}$ for $k\ge 0$, we 
can observe that $\vrho_k \to \infty$ and $\|(\lambda^k, \mu^k)\|/\vrho_k \to 0$. 
We also know that $\{x^k\}_{k \in K} \to x^*$, $\{g(x^k)\}_{k \in K} \to g(x^*)$ and 
$\{h(x^k)\}_{k \in K} \to h(x^*)$. Using these results, and upon taking limits as 
$k \in K \to \infty$ on both sides of the above inequality, we obtain that 
\[
\|[g(x^*)]^+\|^2 + \|h(x^*)\|^2 \le 0,
\] 
which implies that $g(x^*) \le 0$ and $h(x^*)=0$. We also know that $x^* \in X$. 
It thus follows that statement (a) holds. 

We next show that statement (b) also holds. Using \eqnok{inner-prob}, \eqnok{aug}, 
\eqnok{wx}, \eqnok{lag-update}, and the first relation in \eqnok{inner-cond}, we 
have 
\beq \label{inner-term}
\|\nabla f(x^k) + (\lambda^{k+1})^T \nabla g(x^k) + (\mu^{k+1})^T \nabla h(x^k) 
+ z^k + v^k \| \le \epsilon_k        
\eeq
for some $z^k \in \partial \q(x^k)$ and $v^k \in N_X(x^k)$. Suppose for 
contradiction that the subsequence $\{(\lambda^{k+1},\mu^{k+1})\}_{k\in K}$ 
is unbounded. By passing to a subsequence if necessary, we can assume that 
$\{(\lambda^{k+1},\mu^{k+1})\}_{k\in K} \to \infty$. Let $(\bar \lambda^{k+1}, 
\bar \mu^{k+1}) = (\lambda^{k+1}, \mu^{k+1})/{\|(\lambda^{k+1}, \mu^{k+1})\|}$ 
and $\bar v^k = v^k/{\|(\lambda^{k+1},\mu^{k+1})\|}$. Recall that $\{x^k\}_{k\in K} 
\to x^*$. It together with Theorem 6.2.7 of \cite{HiLe93-1} implies that 
$\cup_{k\in K} \partial \q(x^k)$ is bounded, and so is $\{z^k\}_{k\in K}$. 
In addition, $\{g(x^k)\}_{k \in K} \to g(x^*)$ and $\{h(x^k)\}_{k \in K} \to h(x^*)$.
Then, we can observe from \eqnok{inner-term} that $\{\bar v^k\}_{k\in K}$ is bounded. 
Without loss of generality, assume that $\{(\bar \lambda^{k+1}, \bar \mu^{k+1})\}_{k\in K} 
\to (\bar\lambda, \bar\mu)$ and $\{\bar v^k\}_{k\in K} \to \bar v$ (otherwise, one can 
consider their convergent subsequences). Clearly, $\|(\bar\lambda,\bar\mu)\|=1$. Dividing 
both sides of \eqnok{inner-term} 
by $\|(\lambda^{k+1}, \mu^{k+1})\|$ and taking limits as $k\in k \to \infty$, we 
obtain that 
\beq \label{lim-eq}
\bar \lambda^T \nabla g(x^*) + \bar\mu^T \nabla h(x^*) + \bar v = 0. 
\eeq     
Further, using the identity $\lambda^{k+1} = [\lambda^k + \vrho_k g(x^k)]^+$ and 
the fact that $\vrho_k \to \infty$ and $\|\lambda^k\|/\vrho_k \to 0$, we observe 
that $\lambda^{k+1}\in\Re^m_+$ and $\lambda^{k+1}_i = 0$ for $i\notin\cA(x^*)$ 
when $k\in K$ is sufficiently large, which imply that $\bar\lambda\in\Re^m_+$ and 
$\bar\lambda_i=0$ for $i\notin\cA(x^*)$. Moreover, we have $\bar v\in N_X(x^*)$ 
since $N_X(x^*)$ is a closed cone.  Using these results, \eqnok{lim-eq}, Robinson's 
condition \eqnok{rob-cond} at $x^*$, and a similar argument as that in 
the proof of Theorem \ref{opt-thm}, we can obtain that $(\bar\lambda,\bar\mu)=(0,0)$, 
which contradicts the identity $\|(\bar\lambda,\bar\mu)\|=1$. Therefore, the subsequence 
$\{(\lambda^{k+1},\mu^{k+1})\}_{k\in K}$ is bounded. Using this result together with 
\eqnok{inner-term} and the fact $\{z^k\}_{k\in K}$ is bounded, we immediately 
see that $\{v^k\}_{k\in K}$ is bounded. Using semicontinuity of $\partial 
\q(\cdot)$ and $N_X(\cdot)$ (see Theorem $24.4$ of \cite{Rock70} and Lemma $2.42$ 
of \cite{Rus06}), and the fact $\{x^k\}_{k\in K} \to x^*$, we conclude 
that every accumulation point of $\{z^k\}_{k\in K}$ and $\{v^k\}_{k\in K}$ belongs to 
$\partial \q(x^*)$ and $N_X(x^*)$, respectively. Using these results and 
\eqnok{inner-term}, we further see that for every accumulation point $(\lambda^*,\mu^*)$ 
of $\{(\lambda^{k+1},\mu^{k+1})\}_{k \in K}$, there exists some $z^*\in\partial \q(x^*)$ 
and $v^* \in N_X(x^*)$ such that 
\[
\nabla f(x^*) + (\lambda^*)^T \nabla g(x^*) + (\mu^*)^T \nabla h(x^*) 
+ z^* + v^* = 0.
\]   
Moreover, using the identity $\lambda^{k+1} = [\lambda^k + \vrho_k g(x^k)]^+$ and 
the fact that $\vrho_k \to \infty$ and $\|\lambda^k\|/\vrho_k \to 0$, we easily 
see that $\lambda^*\in\Re^m_+$ and $\lambda^*_i=0$ for $i\notin\cA(x^*)$. 
Thus, $(\lambda^*,\mu^*)$ satisfies the first-order optimality 
conditions \eqnok{1st-opt}-\eqnok{complement} at $x^*$. 
\end{proof}

\vgap

Before ending this subsection, we now briefly discuss how to find an approximate 
solution $x^k\in X$ for subproblem \eqnok{inner-prob} satisfying \eqnok{inner-cond} 
as required in Step 1) of the above augmented Lagrangian method. In particular, we 
are interested in applying the nonmonotone gradient methods proposed in Subsection 
\ref{grad-method} to \eqnok{inner-prob}. As shown in Subsection \ref{grad-method} 
(see Theorems \ref{converg-1} and \ref{converg-2}), these methods are able to find 
an approximate solution $x^k\in X$ satisfying the first relation of \eqnok{inner-cond}. 
Moreover, if an initial point for these methods is properly chosen, the obtained 
approximate solution $x^k$ also satisfies the second relation of \eqnok{inner-cond}. 
For example, given $k \ge 0$, let $x^k_\init\in X$ denote the initial point for 
solving the $k$th subproblem \eqnok{inner-prob}, and we define $x^k_\init$ for 
$k \ge 1$ as follows
\[
x^k_\init = \left\{
\ba{ll}
x^{\rm feas}, \ \ \ \mbox{if} \ L_{\vrho_k} (x^{k-1},\lambda^k,\mu^k) > \Upsilon; \\
x^{k-1},  \ \ \ \mbox{otherwise},   
\ea \right.
\]
where $x^{k-1}$ is the approximate solution to the $(k-1)$th subproblem \eqnok{inner-prob} 
satisfying \eqnok{inner-cond} (with $k$ replaced by $k-1$).  Recall from Assumption 
\ref{assum-aug} that $x^{\rm feas}$ is a feasible solution of \eqnok{nlp}. Thus, 
$g(x^{\rm feas}) \le 0$, and $h(x^{\rm feas})=0$, which together with \eqnok{aug}, 
\eqnok{wx} and the definition of $\Upsilon$ implies that 
\[
L_{\vrho_k} (x^{\rm feas},\lambda^k,\mu^k) \ \le \  f(x^{\rm feas}) \ \le \ \Upsilon.
\]
It follows from this inequality and the above choice of $x^k_\init$ that 
$L_{\vrho_k} (x^k_\init,\lambda^k,\mu^k) \le \Upsilon$. Additionally, the nonmonotone 
gradient methods proposed in Subsection \ref{grad-method} possess a natural property 
that the objective function values at all subsequent iterates are bounded below by the 
one at the initial point. Therefore, we have
\[
L_{\vrho_k} (x^k,\lambda^k,\mu^k) \ \le  \ L_{\vrho_k} (x^k_\init,\lambda^k,\mu^k) 
\ \le \ \Upsilon,  
\]   
and so the second relation of \eqnok{inner-cond} is satisfied at $x^k$.

\subsection{Nonmonotone gradient methods for nonsmooth minimization}
\label{grad-method}

In this subsection we propose two nonmonotone gradient methods for minimizing 
a class of nonsmooth functions over a closed convex set, which can be suitably 
applied to the subproblems arising in our augmented Lagrangian method detailed 
in Subsection \ref{aug-method}. We also establish global convergence and 
local linear rate of convergence for these methods. It shall be mentioned that 
these two methods are closely related to the ones proposed in \cite{WrNoFi09} 
and \cite{TseYun09}, but they are not the same (see the remarks below for details). 
In addition, these methods can be viewed as an extension of the well-known 
projected gradient methods studied in \cite{BiMaRa00} for smooth problems, 
but the methods proposed in \cite{WrNoFi09} and \cite{TseYun09} cannot.
       
Throughout this subsection, we consider the following problem
\beq \label{nonsmooth}
\min_{x\in X} \{F(x) := f(x) + \q(x) \}, 
\eeq
where $f: \Re^n \to \Re$ is continuously differentiable, $\q: \Re^n \to \Re$ 
is convex but not necessarily smooth, and $X \subseteq \Re^n$ 
is closed and convex.

We say that $x\in \Re^n$ is a {\it stationary point} of problem \eqnok{nonsmooth} 
if $x\in X$ and 
\beq \label{station}
0 \in \nabla f(x) + \partial \q(x) + N_X(x).
\eeq 

Given a point $x\in \Re^n$ and $H \succ 0$, we denote by $d_H(x)$  the solution of the 
following problem:
\beq \label{dir}
d_H(x) := \arg\min_d \left\{\nabla f(x)^T d + \frac{1}{2} 
d^T H d + \q(x+d): \ x+d \in X \right\}. 
\eeq 

The following lemma provides an alternative characterization of stationarity that 
will be used in our subsequent analysis.

\begin{lemma} \label{opt-char}
For any $H \succ 0$, $x \in X$ is a stationary point of problem 
\eqnok{nonsmooth} if and only if $d_{H}(x)=0$.
\end{lemma} 

\begin{proof}
We first observe that \eqnok{dir} is a convex problem, and moreover its 
objective function is strictly convex. The conclusion of this lemma immediately 
follows from this observation and the first-order optimality condition of \eqnok{dir}.   
\end{proof}

\vgap

The next lemma shows that $\|d_{H}(x)\|$ changes not too fast with $H$. It will be 
used to prove Theorems \ref{lconverg1} and \ref{lconverg2}.

\begin{lemma} \label{dchg-H}
For any $x\in\Re^n$, $H \succ 0$, and $\tilde H \succ 0$, let $d=d_{H}(x)$ 
and $\td=d_{\tH}(x)$. Then 
\beq
\|\td\| \le \frac{1+\lmax(Q)+\sqrt{1-2\lmin(Q)+\lmax(Q)^2}}{2\lmin(\tH)}\lmax(H)\|d\|,
\eeq
where $Q=H^{-1/2}\tH H^{-1/2}$. 
\end{lemma}

\begin{proof}
The conclusion immediately follows from Lemma 3.2 of \cite{TseYun09} with 
$J =\{1,\ldots,n\}$, $c=1$, and $P(x) := P(x)+I_X(x)$, where $I_X$ is the 
indicator function of $X$.   
\end{proof}

\vgap

The following lemma will be used to prove Theorems \ref{lconverg1} and 
\ref{lconverg2}. 

\begin{lemma} \label{step-bdd}
Given $x\in\Re^n$ and $H \succ 0$, let $g=\nabla f(x)$ and 
$\Delta_d = g^Td+\q(x+d)-\q(x)$ for all $d\in\Re^n$. Let $\sigma\in (0,1)$ 
be given. The following statements hold:
\bi
\item[(a)] If $d=d_H(x)$, then
\[
-\Delta_d \ \ge \ d^T H d \ \ge \ \lmin(H) \|d\|^2. 
\] 
\item[(b)]
For any $\bar x \in \Re^n$, $\alpha\in(0,1]$, $d=d_{H}(x)$, and 
$x'=x+\alpha d$, then 
\[
(g+Hd)^T(x'-\bar x) + \q(x') - \q(\bar x) \le (\alpha-1)(d^THd+\Delta_d).
\] 
\item[(c)]
If $f$ satisfies 
\beq \label{lip-cond}
\|\nabla f(y) - \nabla f(z)\| \le L \|y-z\| \ \ \ \forall y, z \in \Re^n 
\eeq
for some $L>0$, then the descent condition 
\[
F(x+\alpha d) \le  F(x) + \sigma \alpha \Delta_d
\]
is satisfied for $d=d_{H}(x)$, provided $0 \le \alpha \le 
\min\{1,2(1-\sigma)\lmin(H)/L\}$. 
\item[(d)] 
If $f$ satisfies \eqnok{lip-cond}, then the descent condition 
\[
F(x+d) \le F(x) + \sigma \Delta_d 
\]  
is satisfied for $d=d_{H(\alpha)}(x)$, where $H(\alpha)=\alpha H$, provided 
$\alpha \ge L/[2(1-\sigma)\lmin(H)]$. 
\ei 
\end{lemma}

\begin{proof} 
The statements (a)-(c) follow from Theorem 4.1 (a) and Lemma 3.4 of \cite{TseYun09} 
with $J=\{1,\ldots,n\}$, $\gamma=0$, and $\underline\lambda=\lmin(H)$. We now 
prove statement (d). Let $d=d_{H(\alpha)}(x)$, where $H(\alpha)=\alpha H$ 
for $\alpha >0$. It follows from statement (a) that 
$\|d\|^2 \le -\Delta_d/(\alpha\lmin(H))$. Using this relation, \eqnok{lip-cond}, 
and the definitions of $F$ and $\Delta_d$, we have 
\beqas
F(x+d) - F(x) &=& f(x+d)-f(x)+\q(x+d)-\q(x) \\
&=& \nabla f(x)^Td + \q(x+d)-\q(x) + \int^1_0 d^T(\nabla f(x+td)-\nabla f(x)) dt 
\\ [4pt]
& \le & \Delta_d + \int^1_0 \|\nabla f(x+td)-\nabla f(x)\|\|d\| dt \\
& \le & \Delta_d + L\|d\|^2/2 \ \le \ \left[1-L/(2\alpha\lmin(H))\right]\Delta_d,
\eeqas
which together with $\alpha \ge L/[2(1-\sigma)\lmin(H)]$, immediately implies that 
statement (d) holds.  
\end{proof}

\vgap

We now present the first nonmonotone gradient method for \eqnok{nonsmooth} as follows.

\gap 

\noindent
\begin{minipage}[h]{6.6 in}
{\bf Nonmonotone gradient method I:} \\ [5pt]
Choose parameters $\eta >1$, $0<\sigma < 1$, $0 < \underline\alpha < \bar\alpha$, 
$0 < \underline\lambda \le \bar\lambda$, and integer $M \ge 0$. Set $k=0$ and 
choose $x^0\in X$. 
\begin{itemize}
\item[1)] Choose $\alpha_k^0 \in [\underline\alpha, \bar\alpha]$ and 
$\underline\lambda I \preceq H_k \preceq \bar\lambda I$. 
\item[2)] For $j=0,1, \ldots$
  \bi
    \item[2a)] Let $\alpha_k = \alpha_k^0 \eta^j$. Solve \eqnok{dir} with $x=x^k$ 
               and $H=\alpha_kH_k$ to obtain $d^k = d_{H}(x)$.  
    \item[2b)] If $d^k$ satisfies 
          \beq \label{reduct}
           F(x^k + d^k) \ \le \ \max_{[k-M]^+ \le i \le k} F(x^i) + \sigma \Delta_k, 
          \eeq  
          go to step 3), where 
          \beq \label{delta}
          \Delta_k := \nabla f(x^k)^T d^k + \q(x^k+d^k) - \q(x^k).
          \eeq
  \ei
\item[3)] Set $x^{k+1} = x^k + d^k$ and $k \leftarrow k+1$. 
\end{itemize}
\noindent
{\bf end}
\end{minipage}

\gap

\begin{remark}
The above method is closely related to the one proposed in \cite{WrNoFi09}. They differ 
from each other only in that the distinct $\Delta_k$'s are used in \eqnok{reduct}. Indeed, 
the latter method uses $\Delta_k = -\alpha_k \|d^k\|^2/2$. Nevertheless, for 
global convergence, the method \cite{WrNoFi09} needs a strong assumption that the 
function $f$ is {\it Lipschitz} continuously differentiable, which is not required 
for our method (see Theorem \ref{converg-1}). In addition, our method can be viewed 
as an extension of the first projected gradient method (namely, SPG1) studied in 
\cite{BiMaRa00} for smooth problems, but their method cannot. Finally, 
local convergence is established for our method (see Theorem \ref{lconverg1}), but 
not studied for their method. 
\end{remark}

\gap

We next prove global convergence of the nonmonotone gradient method I. Before proceeding, 
we establish two technical lemmas below. The first lemma shows that if $x^k\in X$ 
is a nonstationary point, there exists an $\alpha_k >0$ in step 2a) so that 
\eqnok{reduct} is satisfied, and hence the above method is well defined.     

\begin{lemma} \label{alpha-exist}
Suppose that $H_k \succ 0$ and $x^k\in X$ is a nonstationary point of problem 
\eqnok{nonsmooth}. Then, there exists $\tilde\alpha >0$ such that 
$d^k = d_{H_k(\alpha_k)}(x^k)$, where $H_k(\alpha_k)=\alpha_k H_k$, 
satisfies \eqnok{reduct} whenever $\alpha_k \ge \tilde\alpha$. 
\end{lemma}

\begin{proof}
For simplicity of notation, let $d(\alpha) = d_{H_k(\alpha)}(x^k)$, 
where $H_k(\alpha)=\alpha H_k$ for any $\alpha>0$. Then, it follows from 
\eqnok{dir} that for all $\alpha >0$, 
\beq \label{bound-ad}
\alpha \|d(\alpha)\|  \le  -\frac{2[\nabla f(x^k)^T d(\alpha)+\q(x^k+d(\alpha))-\q(x^k)]}
{\lambdam(H_k)\|d(\alpha)\|}  
\le  -\frac{2F'(x^k,d(\alpha)/\|d(\alpha)\|)}{\lambdam(H_k)}.
\eeq
Thus, we easily see that the set $\tilde S : = \{\alpha \|d(\alpha)\|: \alpha >0\}$ 
is bounded. It implies that $\|d(\alpha)\| \to 0$ as $\alpha \to \infty$. We claim 
that 
\beq \label{alphad1}
\liminf_{\alpha \to \infty} \alpha \|d(\alpha)\| \ > \ 0.
\eeq
Suppose not. Then there exists a sequence $\{\balpha_l\} \uparrow \infty$ 
such that $\balpha_l \|d(\balpha_l)\| \to 0$ as $l \to \infty$. Invoking that 
$d(\balpha_l)$ is the optimal solution of \eqnok{dir} with $x=x^k$, $H=\balpha_lH_k$ 
and $\alpha = \balpha_l$, we have 
\[
0 \in \nabla f(x^k) + \balpha_l H_k d(\balpha_l) + \partial \q(x^k+d(\balpha_l)) 
+ N_X(x^k+d(\balpha_l)).
\]
Upon taking limits on both sides as $l \to \infty$, and using semicontinuity 
of $\partial \q(\cdot)$ and $N_X(\cdot)$ (see Theorem $24.4$ 
of \cite{Rock70} and Lemma $2.42$ of \cite{Rus06}), and the relations 
$\|d(\balpha_l)\| \to 0$ and $\balpha_l \|d(\balpha_l)\| \to 0$, we see 
that \eqnok{station} holds at $x^k$, which contradicts the nonstationarity 
of $x^k$. Hence, \eqnok{alphad1} holds. We observe that 
\[
\alpha d(\alpha)^T H_k d(\alpha) \ge \lmin(H_k) \alpha \|d(\alpha)\|^2,
\]
which together with \eqnok{alphad1} and $H_k \succ 0$, implies that 
$\|d(\alpha)\| = O\left(\alpha d(\alpha)^T H_k d(\alpha)\right)$ as 
$\alpha \to \infty$. In addition, using the boundedness of $\tilde S$ 
and $H_k \succ 0$, we have $\alpha d(\alpha)^T H_k d(\alpha) = 
O\left(\|d(\alpha)\|\right)$ as $\alpha \to \infty$. Thus, we obtain 
that 
\beq \label{eq-bdd}
\alpha d(\alpha)^T H_k d(\alpha) = \Theta(\|d(\alpha)\|) \ \ \ \mbox{as} \ 
\alpha \to \infty.
\eeq
%This identity together with the fact that $d(\alpha)$ 
%is the optimal solution of \eqnok{dir} with $x=x^k$ and $H=H_k$, implies that 
This relation together with Lemma \ref{step-bdd}(a) implies that 
\beq \label{delta-bdd}
\q(x^k) - \nabla f(x^k)^T d(\alpha) - \q(x^k+d(\alpha)) \ge \alpha 
d(\alpha)^T H_k d(\alpha) = \Theta(\|d(\alpha)\|).
\eeq
Using this result, and the relation $\|d(\alpha)\| \to 0$ as $\alpha \to \infty$, 
we further have 
\beqa
F(x^k+d(\alpha)) - \max_{[k-M]^+ \le i \le k} F(x^i) &\le & 
F(x^k+d(\alpha)) - F(x^k) \nn \\ 
&=& f(x^k+d(\alpha)) - f(x^k) + \q(x^k+d(\alpha)) - \q(x^k) \nn \\ [4pt]
&=& \nabla f(x^k)^T d(\alpha) + \q(x^k+d(\alpha)) - \q(x^k) + o(\|d(\alpha)\|) \nn \\ [4pt]
& \le & \sigma [\nabla f(x^k)^T d(\alpha) + \q(x^k+d(\alpha)) - \q(x^k)], \label{step-exist}   
\eeqa
provided $\alpha$ is sufficiently large. It implies that the conclusion holds.  
\end{proof}

\vgap

The following lemma shows that the search directions $\{d^k\}$ approach zero, 
and the sequence of objective function values $\{F(x^k)\}$ also converges.  

\begin{lemma} \label{seq-converg1} 
Suppose that $F$ is bounded below in $X$ and uniformly continuous in the 
the level set $\cL=\{x\in X: F(x) \le F(x^0)\}$. Then, the sequence 
$\{x^k\}$ generated by the nonmonotone gradient method I satisfies 
$\lim_{k\to \infty} d^k =0$. Moreover, the sequence $\{F(x^k)\}$ converges.   
\end{lemma}

\begin{proof}
We first observe that $\{x^k\} \subseteq \cL$. Let $l(k)$ be an integer such 
that $[k-M]^+ \le l(k) \le k$ and 
\[
F(x^{l(k)}) = \max \{F(x^i): [k-M]^+ \le i \le k\}
\]
for all $k \ge 0$. We clearly observe that $F(x^{k+1}) \le F(x^{l(k)})$ for all $k \ge 0$, 
which together with the definition of $l(k)$ implies that the sequence $\{F(x^{l(k)})\}$ 
is monotonically nonincreasing. Further, since $F$ is bounded below in $X$, we have 
\beq \label{F-limit}
\lim_{k\to \infty} F(x^{l(k)}) = F^*
\eeq
for some $F^* \in \Re$. We next prove by induction that the following limits 
hold for all $j \ge 1$: 
\beq \label{2-limits}
\lim_{k \to \infty} d^{l(k)-j} = 0, \ \ \ \ \lim_{k \to \infty} F(x^{l(k)-j}) 
= F^*.
\eeq
Using \eqnok{reduct} and \eqnok{delta} with $k$ replaced by $l(k)-1$, we obtain 
that 
\beq \label{reduct-F1}
F(x^{l(k)}) \le F(x^{l(l(k)-1)}) + \sigma \Delta_{l(k)-1}.
\eeq
Replacing $k$ and $\alpha$ by $l(k)-1$ and $\alpha_{l(k)-1}$ in \eqnok{delta-bdd}, 
respectively, and using $H_{l(k)-1} \succeq \underline\lambda I$ and the 
definition of $\Delta_{l(k)-1}$ (see \eqnok{delta}), we have 
\[
\Delta_{l(k)-1} \le -\underline\lambda\alpha_{l(k)-1}\|d^{l(k)-1}\|^2.   
\]
The above two inequalities yield that 
\beq \label{F-reduct}
F(x^{l(k)}) \le F(x^{l(l(k)-1)}) - \sigma\underline\lambda\alpha_{l(k)-1}
\|d^{l(k)-1}\|^2,
\eeq  
which together with \eqnok{F-limit} implies that $\lim_{k\to\infty} 
\alpha_{l(k)-1}\|d^{l(k)-1}\|^2=0$. Further, noticing that $\alpha_k 
\ge \underline\alpha$ for all $k$, we obtain that  
$\lim_{k \to \infty} d^{l(k)-1} = 0$. Using this result and \eqnok{F-limit}, 
we have 
\beq \label{reduct-1}
\lim_{k \to \infty} F(x^{l(k)-1}) = \lim_{k \to \infty} F(x^{l(k)}-d^{l(k)-1}) 
= \lim_{k \to \infty} F(x^{l(k)}) = F^*, 
\eeq 
where the second equality follows from uniform continuity of $F$ in $\cL$. 
Therefore, \eqnok{2-limits} holds for $j=1$. We now need to show that if 
\eqnok{2-limits} holds for $j$, then it also holds for $j+1$. Using a similar 
argument as that leading to \eqnok{F-reduct}, we have
\[
F(x^{l(k)-j}) \le F(x^{l(l(k)-j-1)}) - \sigma\underline\lambda\alpha_{l(k)-j-1}
\|d^{l(k)-j-1}\|^2,
\]
which together with \eqnok{F-limit}, the induction assumption $\lim_{k \to \infty} 
F(x^{l(k)-j}) = F^*$, and the fact that $\alpha_{l(k)-j-1} \ge \underline\alpha$ 
for all $k$, yields $\lim_{k \to \infty} d^{l(k)-j-1} = 0$. 
Using this result, the induction assumption $\lim_{k \to \infty} F(x^{l(k)-j}) = F^*$, 
and a similar argument as that leading to \eqnok{reduct-1}, we can show that 
$\lim_{k \to \infty} F(x^{l(k)-j-1}) = F^*$. Hence, \eqnok{2-limits} holds 
for $j+1$. 

Finally, we will prove that $\lim_{k \to \infty} d^k = 0$ and 
$\lim_{k \to \infty} F(x^k) = F^*$. By the definition of $l(k)$, we see that 
for $k \ge M+1$, $k-M-1 = l(k)-j$ for some $1 \le j \le M+1$, which together 
with the first limit in \eqnok{2-limits}, implies that $\lim_{k \to \infty} 
d^k = \lim_{k \to \infty} d^{k-M-1} = 0$. Additionally, we observe that 
\[
x^{l(k)} = x^{k-M-1} + \sum^{\bar l_k}_{j=1} d^{l(k)-j} \ \ \ \forall k \ge M+1,
\]
where $\bar l_k = l(k)-(k-M-1) \le M+1$. Using the above identity, \eqnok{2-limits}, 
and uniform continuity of $F$ in $\cL$, we see that $\lim_{k \to \infty} F(x^k) 
= \lim_{k \to \infty} F(x^{k-M-1}) = F^*$. Thus, the conclusion of this lemma 
holds.   
\end{proof}

\vgap

We are now ready to show that the nonmonotone gradient method I is 
globally convergent.

\begin{theorem} \label{converg-1}
Suppose that $F$ is bounded below in $X$ and uniformly continuous in the 
level set $\cL=\{x\in X: F(x) \le F(x^0)\}$. Then, any accumulation point 
of the sequence $\{x^k\}$ generated by the nonmonotone gradient method 
I is a stationary point of \eqnok{nonsmooth}.  
\end{theorem}

\begin{proof}
Suppose for contradiction that $x^*$ is an accumulation point of $\{x^k\}$ 
that is a nonstationary point of \eqnok{nonsmooth}. Let $K$  
be the subsequence such that $\{x^k\}_{k\in K} \to x^*$. We first
claim that $\{\alpha_k\}_{k\in K}$ is bounded. Suppose not. Then there exists  
a subsequence of $\{\alpha_k\}_{k\in K}$ that goes to $\infty$. Without loss of 
generality, we assume that $\{\alpha_k\}_{k\in K} \to \infty$. 
For simplicity of notations, let $\balpha_k=\alpha_k/\eta$, $d^k(\alpha) 
= d_{H_k(\alpha)}(x^k)$ for $k\in K$ and $\alpha>0$, where $H_k(\alpha)=\alpha H_k$. 
Since $\{\alpha_k\}_{k\in K} \to \infty$ and $\alpha^0_k \le \balpha$, there 
exists some index $\bar k \ge 0$ such that $\alpha_k > \alpha^0_k$ for all 
$k\in K$ with $k \ge \bar k$. By the particular choice of $\alpha_k$ specified 
in steps (2a) and (2b), we have 
\beq \label{low-bdd1}
F(x^k+d^k(\balpha_k)) >  \max_{[k-M]^+ \le i \le k} F(x^i) + \sigma 
[\nabla f(x^k)^T d^k(\balpha_k) + \q(x^k+d^k(\balpha_k)) - \q(x^k)],   
\eeq  
Using a similar argument as that leading to \eqnok{bound-ad}, we have
\[
\balpha_k \|d^k(\balpha_k)\|  \le  -\frac{2F'(x^k,d^k(\balpha_k)/\|d^k(\balpha_k)\|)}
{\lambdam(H_k)} \ \ \ \forall k \in K,
\]
which along with the relations $H_k \succeq \underline\lambda I$ and 
$\{x^k\}_{k\in K} \to x^*$, implies that 
$\{\balpha_k \|d^k(\balpha_k)\|\}_{k\in K}$ is bounded. Since $\{\balpha_k\}_{k\in K} 
\to \infty$, we further have $\{\|d^k(\balpha_k)\|\}_{k\in K} \to 0$ . We now 
claim that 
\beq \label{balphad}
\liminf_{k \in K, k \to \infty} \balpha_k \|d^k(\balpha_k)\| \ > \ 0.
\eeq
Suppose not. By passing to a subsequence if necessary, we can assume that 
$\{\balpha_k \|d^k(\balpha_k)\|\}_{k\in K} \to 0$. Invoking that $d^k(\balpha_k)$ 
is the optimal solution of \eqnok{dir} with $x=x^k$ and $H=\balpha_k H_k$,
we have 
\[
0 \in \nabla f(x^k) + \balpha_k H_k d^k(\balpha_k) + \partial \q(x^k+d^k(\balpha_k)) 
+ N_X(x^k+d^k(\balpha_k))  \ \ \ \forall k \in K.
\]
Upon taking limits on both sides as $k \in K \to \infty$, and using 
semicontinuity of $\partial \q(\cdot)$ and $N_X(\cdot)$ (see Theorem $24.4$ 
of \cite{Rock70} and Lemma $2.42$ of \cite{Rus06}), the relations 
$\underline\lambda I \preceq H_k \preceq \bar\lambda I$, 
$\{\|d^k(\balpha_k)\|\}_{k\in K} \to 0$, $\{\balpha_k 
\|d^k(\balpha_k)\|\}_{k\in K} \to 0$ and $\{x^k\}_{k \in K} \to x^*$, we 
see that \eqnok{station} holds at $x^*$, which contradicts 
nonstationarity of $x^*$.Thus, \eqnok{balphad} holds. Now, using \eqnok{balphad}, 
the relation $H_k \succeq \underline\lambda I$, boundedness of $\{\balpha_k 
\|d^k(\balpha_k)\|\}_{k\in K}$, and a similar argument as that leading to 
\eqnok{eq-bdd}, we observe that $\balpha_k d^k(\balpha_k)^T 
H_k d^k(\balpha_k) = \Theta(\|d^k(\balpha_k)\|)$ as $k \in K \to \infty$. 
Using this result and a similar argument as that leading to \eqnok{step-exist}, we have
\[
F(x^k+d^k(\balpha_k)) \le  \max_{[k-M]^+ \le i \le k} F(x^i) + \sigma 
[\nabla f(x^k)^T d^k(\balpha_k) + \q(x^k+d^k(\balpha_k)) - \q(x^k)],   
\]
provided that $k\in K$ is sufficiently large. The above inequality evidently 
contradicts \eqnok{low-bdd1}. Thus, $\{\alpha_k\}_{k\in K}$ is bounded.

Finally, invoking that $d^k=d^k(\alpha_k)$ is the optimal solution of 
\eqnok{dir} with $x=x^k$, $H=\alpha_k H_k$, we have 
\beq \label{dk-opt}
0 \in \nabla f(x^k) + \alpha_k H_k d^k + \partial \q(x^k+d^k) 
+ N_X(x^k+d^k)  \ \ \ \forall k \in K.
\eeq
By Lemma \ref{seq-converg1}, we have $\{d^k\}_{k\in K} \to 0$. Upon 
taking limits on both sides of \eqnok{dk-opt} as $k \in K \to \infty$, 
and using semicontinuity of $\partial \q(\cdot)$ and $N_X(\cdot)$ (see 
Theorem $24.4$ of \cite{Rock70} and Lemma $2.42$ of \cite{Rus06}), and 
the relations $\underline\lambda I \preceq H_k \preceq \bar\lambda I$, 
$\{d^k\}_{k\in K} \to 0$ and $\{x^k\}_{k \in K} \to x^*$, we see that 
\eqnok{station} holds at $x^*$, which contradicts the nonstationarity 
of $x^*$ that is assumed at the beginning of this proof. Therefore, the 
conclusion of this theorem holds.
\end{proof}

\vgap

We next analyze the asymptotic convergence rate of the nonmonotone 
gradient method I under the following assumption, which is the same as 
that made in \cite{TseYun09}. In what follows, we denote by $\bar X$ the 
set of stationary points of problem \eqnok{nonsmooth}. 

\begin{assumption} \label{assump}
\bi
\item[(a)] $\bar X \neq \emptyset$ and, for any $\zeta \ge \min_{x\in X} F(x)$, 
there exists $\tau>0$ and $\epsilon>0$ such that 
\[
\dist(x,\bar X) \le \tau \|d_I(x)\| \ \ \ \mbox{whenever} \ \ \ F(x) \le \zeta, \ 
\|d_I(x)\| \le \epsilon. 
\] 
\item[(b)]
There exists $\delta>0$ such that 
\[
\|x-y\| \ge \delta \ \ \ \mbox{whenever} \ \ \ x\in\bar X, y \in \bar X, F(x) \neq F(y). 
\]
\ei
\end{assumption}

\vgap

We are ready to establish local linear rate of convergence for the nonmonotone 
gradient method I described above. The proof of the following theorem is inspired 
by the work of Tseng and Yun \cite{TseYun09}, who analyzed a similar local convergence 
for a coordinate gradient descent method for a class of nonsmooth minimization problems.    

\begin{theorem} \label{lconverg1}
Suppose that $f$ satisfies \eqnok{lip-cond}, and $F$ is bounded below in $X$ 
and uniformly continuous in the level set $\cL=\{x\in X: F(x) \le F(x^0)\}$.  
Then, the sequence $\{x^k\}$ generated by the nonmonotone gradient method I 
satisfies 
\[
F(x^{l(k)}) - F^* \ \le \ c (F(x^{l(l(k))-1)}-F^*),
\] 
provided $k$ is sufficiently large, where $F^* = \lim_{k \to \infty} F(x^k)$ 
(see Lemma \ref{seq-converg1}), and $c$ is some constant in $(0,1)$.
\end{theorem}

\begin{proof}
Invoking $\alpha^0_k \le \balpha$ and the specific choice of $\alpha_k$, 
we see from Lemma \ref{step-bdd}(d) that $\hat\alpha := \sup_k \alpha_k < \infty$. 
Let $H_k(\alpha) = \alpha H_k$. Then, it follows from $\underline \lambda I 
\preceq H_k \preceq \bar\lambda I$ and $\alpha_k \ge \underline\alpha$ that 
$(\underline \alpha \cdot \underline\lambda) I \preceq H_k(\alpha_k) \preceq 
\hat\alpha\bar\lambda I$. Using this relation, Lemma \ref{dchg-H}, $H_k \succeq  
{\underline \lambda} I$, and $d^k=d_{H_k(\alpha_k)}(x^k)$, we obtain that 
\beq \label{d-order}
\|d_I(x^k)\| = \Theta(\|d^k\|),
\eeq
which together with Lemma \ref{seq-converg1} implies $\{d_I(x^k)\} \to 0$. 
Thus, for any $\epsilon >0$, there exists some index $\bar k$ such that 
$d_I(x^{l(k)-1}) \le \epsilon$ for all $k \ge \bar k$. In addition, we 
clearly observe that $F(x^{l(k)-1}) \le F(x^0)$. Then, by Assumption 
\ref{assump}(a) and \eqnok{d-order}, there exists some index $k'$ such 
that 
\beq \label{diff-bdd}
\|x^{l(k)-1}-\bar x^{l(k)-1}\| \le c_1 \|d^{l(k)-1}\| \ \ \ \forall k \ge k'
\eeq
for some $c_1 >0$ and $\bar x^{l(k)-1}\in\bar X$. Note that 
\[
\|x^{l(k+1)-1} - x^{l(k)-1}\| \le \sum_{i=l(k)-1}^{l(k+1)-2} \|d^i\| \le  
\sum_{i=[k-M-1]^+}^{[k-1]^+} \|d^i\|,      
\]
which together with $\{d^k\} \to 0$, implies that 
$\|x^{l(k+1)-1} - x^{l(k)-1}\| \to 0$. Using this result, \eqnok{diff-bdd}, 
and Lemma \ref{seq-converg1}, we obtain 
\[
\ba{lcl}
\|\bar x^{l(k+1)-1} - \bar x^{l(k)-1}\| &\le & \|x^{l(k+1)-1}-\bar x^{l(k+1)-1}\|
+ \|x^{l(k)-1}-\bar x^{l(k)-1}\| + \|x^{l(k+1)-1}-\bar x^{l(k)-1}\| \\ [5pt]
& \le & c_1 \|d^{l(k+1)-1}\| + c_1 \|d^{l(k)-1}\| + \|x^{l(k+1)-1}-\bar x^{l(k)-1}\| 
\ \to \ 0.
\ea
\]
It follows from this relation and Assumption \ref{assump}(b) that there exists 
an index $\hat k \ge k'$ and $v \in \Re$ such that 
\beq \label{eq-f}
F(\bar x^{l(k)-1})= v \ \ \ \forall k \ge \hat k.
\eeq
Then, by Lemma 5.1 of \cite{TseYun09}, we see that
\beq \label{liminf-F}
F^* = \lim_{k\to\infty} F(x^k) = \liminf_{k\to\infty} F(x^{l(k)-1}) \ \ge \ v.
\eeq 
Further, using the definition of $F$, \eqnok{lip-cond}, \eqnok{eq-f}, 
Lemma \ref{step-bdd}(b), and $H_k(\alpha_k) \preceq \hat\alpha \bar\lambda I$, we 
have for $k \ge \hat k$, 
\beqa
F(x^{l(k)}) - v &=& f(x^{l(k)}) +\q(x^{l(k)}) - f(\bx^{l(k)-1}) - \q(\bx^{l(k)-1}) 
\nn \\ [4pt]
&=& \nabla f(\tx^k)^T(x^{l(k)}-\bx^{l(k)-1}) +\q(x^{l(k)}) - \q(\bx^{l(k)-1}) \nn \\ [4pt]
&=& (\nabla f(\tx^k)-\nabla f(x^{l(k)-1})^T(x^{l(k)}-\bx^{l(k)-1}) - 
(H_{l(k)-1}(\alpha_{l(k)-1})d^{l(k)-1})^T(x^{l(k)}-\bx^{l(k)-1}) \nn \\ [4pt]
& & + \left[(\nabla f(x^{l(k)-1})+H_{l(k)-1}(\alpha_{l(k)-1})d^{l(k)-1})^T(x^{l(k)}-\bx^{l(k)-1}) 
+\q(x^{l(k)}) - \q(\bx^{l(k)-1})\right] \nn \\ [4pt]
&\le & L \|\tx^k - x^{l(k)-1}\| \|x^{l(k)}-\bx^{l(k)-1}\| + \hat\alpha\bar\lambda \|d^{l(k)-1}\| 
\|x^{l(k)}-\bx^{l(k)-1}\|, \label{loc-bdd1-1}
\eeqa
where $\tx^k$ is some point lying on the segment joining $x^{l(k)}$ with $\bx^{l(k)-1}$. 
It follows from \eqnok{diff-bdd} that, for $k \ge \hat k$,  
\[
\|\tx^k - x^{l(k)-1}\|   \le  \|x^{l(k)}-x^{l(k)-1}\| + \|x^{l(k)-1}-\bx^{l(k)-1}\| 
= (1+c_1) \|d^{l(k)-1}\|.
\]
Similarly, $\|x^{l(k)}-\bx^{l(k)-1}\| \le (1+c_1) \|d^{l(k)-1}\|$ for $k \ge \hat k$. 
Using these inequalities, Lemma \ref{step-bdd}(a), $H_k(\alpha_k) \succeq (\underline
\alpha \cdot \underline \lambda) I$, and \eqnok{loc-bdd1-1}, we see that for $k \ge \hat k$,  
\[
F(x^{l(k)}) - v \ \le \ -c_2 \Delta_{l(k)-1}
\]  
for some constant $c_2 > 0$. This inequality together with \eqnok{reduct-F1} gives
\beq \label{loc-bdd2}
F(x^{l(k)}) - v \ \le \ c_3 \left(F(x^{l(l(k)-1)})-F(x^{l(k)})\right) 
\ \ \ \forall k \ge \hat k,
\eeq
where $c_3 = c_2/\sigma$. Using $\lim_{k\to \infty} F(x^{l(k)})=F^*$, 
and upon taking limits on both sides of \eqnok{loc-bdd2}, we see that 
$F^* \le v$, which together with \eqnok{liminf-F} implies that $v=F^*$. Using this result and 
upon rearranging terms of \eqnok{loc-bdd2}, we have  
\[
F(x^{l(k)}) - F^* \ \le \ c (F(x^{l(l(k))-1)}-F^*) \ \ \ \forall k \ge \hat k,
\]  
where $c=c_3/(1+c_3)$.
\end{proof}

\vgap

We next present the second nonmonotone gradient method for \eqnok{nonsmooth} as follows. 

\gap

\noindent
\begin{minipage}[h]{6.6 in}
{\bf Nonmonotone gradient method II:} \\ [5pt]
Choose parameters $0< \eta < 1$, $0<\sigma < 1$, $0 < \underline\alpha < \bar\alpha$, 
$0 < \underline\lambda \le \bar\lambda$, and integer $M \ge 0$. Set $k=0$ and 
choose $x^0\in X$. 
\begin{itemize}
\item[1)] Choose $\underline\lambda I \preceq H_k \preceq \bar\lambda I$. 
\item[2)] Solve \eqnok{dir} with $x=x^k$ and $H=H_k$ to obtain 
          $d^k = d_{H}(x)$, and compute $\Delta_k$ according to \eqnok{delta}.  
\item[3)] Choose $\alpha^0_k \in [\underline\alpha, \bar\alpha]$. Find the 
          smallest integer $j \ge 0$ such that $\alpha_k = \alpha^0_k \eta^j$ 
          satisfies   
          \beq \label{reduct2}
           F(x^k + \alpha_k d^k) \ \le \ \max_{[k-M]^+ \le i \le k} F(x^i) + 
           \sigma \alpha_k \Delta_k,
          \eeq  
          where $\Delta_k$ is defined in \eqnok{delta}.
\item[4)] Set $x^{k+1} = x^k + \alpha_k d^k$ and $k \leftarrow k+1$. 
\end{itemize}
\noindent
{\bf end}
\end{minipage}

\gap

\begin{remark}
The above method is closely related to the one proposed in \cite{TseYun09}. When 
the entire coordinate block, that is, $J=\{1,\ldots,n\}$ is chosen for the latter 
method, it becomes a special case of our method with $M=0$, which is a gradient 
descent method. Given that our method is generally a nonmonotone method especially 
when $M \ge 1$, most proofs of global and local convergence for the method 
\cite{TseYun09} do not hold for our method directly. In addition, our method can be 
viewed as an extension of the second projected gradient method (namely, SPG2) 
studied in \cite{BiMaRa00} for smooth problems, but the method \cite{TseYun09} 
generally cannot. 
\end{remark}

\gap

We next prove global convergence of the nonmonotone gradient method II. Before 
proceeding, we establish two technical lemmas below. The first lemma shows that 
if $x^k\in X$ is a nonstationary point, there exists an $\alpha_k >0$ in step 3) 
so that \eqnok{reduct2} is satisfied, and hence the above method is well defined.     

\begin{lemma} \label{alpha-exist2}
Suppose that $H_k \succ 0$ and $x^k\in X$ is a nonstationary point of problem 
\eqnok{nonsmooth}. Then, there exists $\tilde\alpha >0$ such that 
$d^k = d_{H_k}(x^k)$ satisfies \eqnok{reduct2} whenever $0<\alpha_k \le 
\tilde\alpha$. 
\end{lemma}

\begin{proof}
In view of Lemma $2.1$ of \cite{TseYun09} with $J=\{1,\ldots,n\}$, $c=1$, 
$x=x^k$, and $H=H_k$, we have 
\[
\ba{lcl}
F(x^k + \alpha d^k)  &\le&   F(x^k) + \alpha \Delta_k + o(\alpha) \\ [4pt]
&\le&  \max\limits_{[k-M]^+ \le i \le k} F(x^i) + \alpha \Delta_k + o(\alpha) \ \ \
\forall \alpha \in (0,1],
\ea
\]
where $\Delta_k$ is defined in \eqnok{delta}. Using the assumption of this 
lemma, we see from Lemma \ref{opt-char} that $d^k \neq 0$, which together 
with $H_k \succ 0$ and Lemma \ref{step-bdd}(a) implies $\Delta_k <0$. The 
conclusion of this lemma immediately follows from this relation and the 
above inequality. 
\end{proof}

\vgap

The following lemma shows that the scaled search directions  $\{\alpha_k d^k\}$ 
approach zero, and the sequence of objective function values $\{F(x^k)\}$ also 
converges.  

\begin{lemma} \label{seq-converg2}
Suppose that $F$ is bounded below in $X$ and uniformly continuous in the 
level set $\cL=\{x\in X: F(x) \le F(x^0)\}$. Then, the sequence 
$\{x^k\}$ generated by the nonmonotone gradient method II satisfies 
$\lim_{k\to \infty} \alpha_k d^k =0$. Moreover, the sequence $\{F(x^k)\}$ 
converges.   
\end{lemma}

\begin{proof}
Let $l(k)$ be defined in the proof of Lemma \ref{seq-converg1}. 
We first observe that $\{x^k\} \subseteq \cL$.  
Using \eqnok{delta}, the definition of $d^k$, and $H_k \succeq \underline\lambda I$, 
we have
\beq \label{delta-lbdd}
\Delta_k  = \nabla f(x^k)^T d^k + \q(x^k+d^k) - \q(x^k) \le -\frac{1}{2} 
(d^k)^T H_k d^k \le -\frac12 \underline\lambda \|d^k\|^2,
\eeq
which together with the relation $\alpha_k \le \alpha^0_k \le \bar\alpha$, 
implies that 
\beq \label{alphad2}
\alpha^2_k \|d^k\|^2 \le -2\bar\alpha\alpha_k \Delta_k/{\underline\lambda}.
\eeq  
By a similar argument as that leading to \eqnok{F-limit}, we see 
that $\{x^k\}$ satisfies \eqnok{F-limit} for some $F^*$. We next show 
by induction that the following limits hold for all $j \ge 1$: 
\beq \label{2-limits2}
\lim_{k \to \infty} \alpha_{l(k)-j} d^{l(k)-j} = 0, \ \ \ \ 
\lim_{k \to \infty} F(x^{l(k)-j}) = F^*.
\eeq
Indeed, using \eqnok{reduct2} with $k$ replaced by $l(k)-1$, we obtain 
that 
\[
F(x^{l(k)}) \le F(x^{l(l(k)-1)}) + \sigma \alpha_{l(k)-1}\Delta_{l(k)-1}.
\]
It together with \eqnok{F-limit} immediately yields 
$\lim_{k\to\infty}\alpha_{l(k)-1}\Delta_{l(k)-1} = 0$. Using this result 
and \eqnok{alphad2}, we see that the first identity of \eqnok{2-limits2} 
holds for $j=1$. Further, in view of this identity, \eqnok{F-limit}, and 
uniform continuity of $F$ in $\cL$, we can easily see that the second 
identity of \eqnok{2-limits2} also holds $j=1$. We now need to show that if 
\eqnok{2-limits2} holds for $j$, then it also holds for $j+1$. First, it 
follows from \eqnok{reduct2} that
\[
F(x^{l(k)-j}) \le F(x^{l(l(k)-j-1)})+ \sigma \alpha_{l(k)-j-1}\Delta_{l(k)-j-1},
\]
which together with \eqnok{F-limit} and the induction assumption 
that $\lim_{k \to \infty} F(x^{l(k)-j}) = F^*$, yields $\lim_{k \to \infty} 
\alpha_{l(k)-j-1} \Delta_{l(k)-j-1} = 0$. Using this result and \eqnok{alphad2}, 
we have $\lim_{k \to \infty} \alpha_{l(k)-j-1} d^{l(k)-j-1} = 0$.  In view of this 
identity, uniform continuity of $F$ in $\cL$ and the induction assumption 
$\lim_{k \to \infty} F(x^{l(k)-j}) = F^*$, we can easily show that 
$\lim_{k \to \infty} F(x^{l(k)-j-1}) = F^*$. Hence, \eqnok{2-limits2} holds for 
$j+1$. The conclusion of this lemma then follows from \eqnok{2-limits2} and 
a similar argument as that in the proof of Lemma \ref{seq-converg2}. 
\end{proof}

\vgap

We are now ready to show that the nonmonotone gradient method II is 
globally convergent.

\begin{theorem} \label{converg-2}
Suppose that $F$ is bounded below in $X$ and uniformly continuous in the 
level set $\cL=\{x\in X: F(x) \le F(x^0)\}$. Then, any accumulation point 
of the sequence $\{x^k\}$ generated by the nonmonotone gradient method 
II is a stationary point of \eqnok{nonsmooth}.  
\end{theorem}

\begin{proof}
Suppose for contradiction that $x^*$ is an accumulation point of $\{x^k\}$ 
that is a nonstationary point of \eqnok{nonsmooth}. Let $K$ be the 
subsequence such that $\{x^k\}_{k\in K} \to x^*$. We first
claim that $\liminf_{k\in K, k\to\infty} \|d^k\| > 0$. Suppose not. By passing 
to a subsequence if necessary, we can assume that $\{\|d^k\|\}_{k\in K} \to 0$. 
Invoking that $d^k$ is the optimal solution of \eqnok{dir} with 
$x=x^k$ and $H=H_k$, we have 
\[
0 \in \nabla f(x^k) + H_k d^k + \partial \q(x^k+d^k) + N_X(x^k+d^k)  \ \ \
\forall k \in K.
\]
Upon taking limits on both sides as $k \in K \to \infty$, and using semicontinuity 
of $\partial \q(\cdot)$ and $N_X(\cdot)$ (see Theorem $24.4$ of \cite{Rock70} and 
Lemma $2.42$ of \cite{Rus06}) the relations 
$\underline\lambda I \preceq H_k \preceq \bar\lambda I$, 
$\{\|d^k\|\}_{k\in K} \to 0$ and $\{x^k\}_{k \in K} \to x^*$, we see that 
\eqnok{station} holds at $x^*$, which contradicts the nonstationarity of $x^*$. 
Thus, $\liminf_{k\in K, k\to\infty} \|d^k\| > 0$ holds. Further, using a similar 
argument as that leading to \eqnok{bound-ad}, we have
\[
\|d^k\|  \le  -\frac{2F'(x^k,d^k/\|d^k\|)}{\lambdam(H_k)} \ \ \ \forall k \in K,
\]
which together with $\{x^k\}_{k\in K} \to x^*$, $H_k \succeq \underline\lambda I$ 
and $\liminf_{k\in K, k\to\infty} \|d^k\| > 0$, implies that 
$\|d^k\| = \Theta(1)$ for $k\in K$. Further, using \eqnok{delta-lbdd}, we see that 
$\lim\sup_{k\in K, k\to\infty} \Delta_k < 0$. Now, it follows from Lemma \ref{seq-converg2} 
and the relation $\liminf_{k\in K, k\to\infty} \|d^k\| > 0$ that 
$\{\alpha_k\}_{k\in K} \to 0$. Since $\alpha^0_k \ge \underline \alpha >0$, there 
exists some index $\bar k \ge 0$ such that $\alpha_k < \alpha^0_k$ and 
$\alpha_k < \eta$ for all $k\in K$ with $k \ge \bar k$. Let $\bar\alpha_k = \alpha_k/\eta$. 
Then, $\{\bar\alpha_k\}_{k\in K} \to 0$ and $0 < \bar\alpha_k \le 1$ for all $k\in K$. 
By the stepsize rule used in step (3), we have, for all $k\in K$ with 
$k \ge \bar k$, 
\beq \label{low-bdd2}
 F(x^k + \bar\alpha_k d^k) > \ \max_{[k-M]^+ \le i \le k} F(x^i) + 
           \sigma \bar\alpha_k \Delta_k,
\eeq
On the other hand, in view of the definition of $F$, \eqnok{delta}, the relations 
$\|d^k\|=\Theta(1)$ and $\lim\sup_{k\in K, k\to\infty} \Delta_k < 0$, and 
the monotonicity of $(\q(x^k + \alpha d^k) - \q(x^k))/\alpha$, we obtain that, for sufficiently 
large $k\in K$,
\beqas
F(x^k + \bar\alpha_k d^k) &=&  f(x^k + \bar\alpha_k d^k) + \q(x^k + \bar\alpha_k d^k) \\ [4pt]
&=& f(x^k + \bar\alpha_k d^k) - f(x^k) + \q(x^k + \bar\alpha_k d^k) - \q(x^k) + F(x^k) \\ [4pt]
& =& \bar\alpha_k \nabla f(x^k)^T d^k + o(\bar\alpha_k\|d^k\|) + 
\q(x^k + \bar\alpha_k d^k) - \q(x^k) + F(x^k) \\ [4pt]
& \le & \bar\alpha_k \nabla f(x^k)^T d^k + o(\bar\alpha_k) + \bar\alpha_k[\q(x^k + d^k) - \q(x^k)] 
+ \max_{[k-M]^+ \le i \le k} F(x^i) \\ [4pt]
&=& \max_{[k-M]^+ \le i \le k} F(x^i) + \bar\alpha_k \Delta_k + o(\bar\alpha_k) \\ [4pt]
& < & \max_{[k-M]^+ \le i \le k} F(x^i) + \sigma \bar\alpha_k \Delta_k, 
\eeqas  
which clearly contradicts \eqnok{low-bdd2}. Therefore, the conclusion of this theorem 
holds.
\end{proof}

\vgap

We next establish local linear rate of convergence for the nonmonotone 
gradient method II described above. The proof of the following theorem is 
inspired by the work of Tseng and Yun \cite{TseYun09}.

\begin{theorem} \label{lconverg2}
Suppose that $\bar\alpha \le 1$, $f$ satisfies \eqnok{lip-cond}, and $F$ 
is bounded below in $X$ and uniformly continuous in the level set 
$\cL=\{x\in X: F(x) \le F(x^0)\}$. Then, the sequence $\{x^k\}$ generated by 
the nonmonotone gradient method II satisfies 
\[
F(x^{l(k)}) - F^* \ \le \ c (F(x^{l(l(k))-1)}-F^*)
\] 
provided $k$ is sufficiently large, where $F^* = \lim_{k \to \infty} F(x^k)$ 
(see Lemma \ref{seq-converg2}), and $c$ is some constant in $(0,1)$.
\end{theorem}

\begin{proof}
Since $\alpha_k$ is chosen by the Armijo rule with $\alpha^0_k \ge \underline\alpha >0$, 
we see from Lemma \ref{step-bdd}(c) that $\inf_k \alpha_k > 0$. It together 
with Lemma \ref{seq-converg2} implies that $\{d^k\} \to 0$. 
Further, using Lemma \ref{dchg-H} and the fact that $d^k=d_{H_k}(x^k)$ and 
$\underline\lambda I \preceq H_k \preceq \bar\lambda I$, we obtain that 
$\|d_I(x^k)\| = \Theta(\|d^k\|)$, and hence $\{d_I(x^k)\} \to 0$. 
%Since $\{F(x^{l(k)})\}$ is nonincreasing, 
%it implies that $F(x^{l(k)-1}) \le F(x^0)$ and $d_I(x^{l(k)-1}) \le \epsilon$ for 
%all $k \ge \bar k$, where $\bar k$ is some index. Then, by Assumption 
%\ref{assump}(a) and \eqnok{d-order}, we have 
Then, by a similar argument as that in the proof of Theorem \ref{lconverg1}, there exist 
$c_1 >0$, $v\in\Re$, and $\bar x^{l(k)-1}\in\bar X$ such that 
\[
\|x^{l(k)-1}-\bar x^{l(k)-1}\| \le c_1 \|d^{l(k)-1}\|, \ \ \ 
F(\bar x^{l(k)-1})= v \ \ \ \forall k \ge \hat k,
\]
where $\hat k$ is some index. 
%for some $c_1 >0$ and $\bar x^{l(k)-1}\in\bar X$. Note that 
%\[
%\|x^{l(k+1)-1} - x^{l(k)-1}\| \le \sum_{i=l(k)-1}^{l(k+1)-2} \|d^i\| \le  
%\sum_{i=[k-M-1]^+}^{[k-1]^+} \|d^i\|,      
%\]
%which together with $\{d^k\} \to 0$, implies that 
%$\|x^{l(k+1)-1} - x^{l(k)-1}\| \to 0$. It follows from this result and Assumption 
%\ref{assump}(b) that there exists an index $\hat k \ge \bar k$ and $v \in \Re$ such 
%that $F(\bar x^{l(k)-1})= v$ for all $k \ge \hat k$. 
Then, by Lemma 5.1 of \cite{TseYun09}, we see that \eqnok{liminf-F} holds for 
$\{x^k\}$, and the above $F^*$ and $v$. 
%\[
%F^* = \lim_{k\to\infty} F(x^k) = \liminf_{k\to\infty} F(x^{l(k)-1}) \ \ge \ v.
%\] 
Further, using the definition of $F$, \eqnok{lip-cond}, Lemma \ref{step-bdd}(b), and 
$\underline\lambda I \preceq H_k \preceq \bar\lambda I$, we have, 
for $k \ge \hat k$, 
\beqa
F(x^{l(k)}) - v &=& f(x^{l(k)}) +\q(x^{l(k)}) - f(\bx^{l(k)-1}) - \q(\bx^{l(k)-1}) 
\nn \\ [4pt]
&=& \nabla f(\tx^k)^T(x^{l(k)}-\bx^{l(k)-1}) +\q(x^{l(k)}) - \q(\bx^{l(k)-1}) \nn \\ [4pt]
&=& (\nabla f(\tx^k)-\nabla f(x^{l(k)-1})^T(x^{l(k)}-\bx^{l(k)-1}) - 
(H_{l(k)-1}d^{l(k)-1})^T(x^{l(k)}-\bx^{l(k)-1}) \nn \\ [4pt]
& & + \left[(\nabla f(x^{l(k)-1})+H_{l(k)-1}d^{l(k)-1})^T(x^{l(k)}-\bx^{l(k)-1}) 
+\q(x^{l(k)}) - \q(\bx^{l(k)-1})\right] \nn \\ [4pt]
&\le & L \|\tx^k - x^{l(k)-1}\| \|x^{l(k)}-\bx^{l(k)-1}\| + \bar\lambda \|d^{l(k)-1}\| 
\|x^{l(k)}-\bx^{l(k)-1}\| \nn \\ [4pt] 
& &+ (\alpha_{l(k)-1}-1)\left[(d^{l(k)-1})^TH_{l(k)-1}d^{l(k)-1}+\Delta_{l(k)-1}\right], 
\label{loc-bdd2-1}
\eeqa
where $\tx^k$ is some point lying on the segment joining $x^{l(k)}$ with $\bx^{l(k)-1}$. 
It follows from \eqnok{diff-bdd} and $\alpha_k \le 1$ that, for $k \ge \hat k$,  
\[
\|\tx^k - x^{l(k)-1}\|   \le  \|x^{l(k)}-x^{l(k)-1}\| + \|x^{l(k)-1}-\bx^{l(k)-1}\| 
\le (1+c_1) \|d^{l(k)-1}\|.
\]
Similarly, $\|x^{l(k)}-\bx^{l(k)-1}\| \le (1+c_1) \|d^{l(k)-1}\|$ for $k \ge \hat k$. 
Using these inequalities, Lemma \ref{step-bdd}(a), $H_k \succeq \underline \lambda I$, 
$\alpha_k \le 1$, and \eqnok{loc-bdd2-1}, we see that, for $k \ge \hat k$,  
\[
F(x^{l(k)}) - v \ \le \ -c_2 \Delta_{l(k)-1}
\]  
for some constant $c_2 > 0$. The remaining proof follows similarly as that of Theorem 
\ref{lconverg1}.
%This inequality together with \eqnok{reduct-1} and 
%$\inf_k \alpha_k>0$ gives
%\beq \label{loc-bdd2}
%F(x^{l(k)}) - v \ \le \ c_3 (F(x^{l(l(k)-1)})-F(x^{l(k)})) \ \ \ \forall k \ge \hat k,
%\eeq
%where $c_3 = c_2/(\sigma\inf_k \alpha_k)$. Using $\lim_{k\to \infty} F(x^{l(k)})=F^*$, 
%and upon taking limits on both sides of \eqnok{loc-bdd2}, we see that 
%$F^* \le v$, which together with \eqnok{liminf-F} implies that $v=F^*$. Using this result, 
%and upon rearranging terms, we have  
%\[
%F(x^{l(k)}) - F^* \ \le \ c (F(x^{l(l(k))-1)}-F^*) \ \ \ \forall k \ge \hat k,
%\]  
%where $c=c_3/(1+c_3)$.
\end{proof}

\section{Augmented Lagrangian method for sparse PCA} 
\label{aug-spca}

In this section we discuss the applicability and implementation details of the
augmented Lagrangian method proposed in Section \ref{aug-nonsmooth} for solving 
sparse PCA \eqnok{diag-approx}. 

\subsection{Applicability of augmented Lagrangian method for \eqnok{diag-approx}}
\label{suitability}

We first observe that problem \eqnok{diag-approx} can be reformulated as
\beq \label{smooth-form}
\ba{rl}
\min\limits_{V \in \Re^{n \times r}} & -\tr(V^T \hsigma V) + \rho \bt |V| \\ [4pt]
\mbox{s.t.} &   V^T_i \hsigma V_j \le \Delta_{ij}  \  \  \ \ \forall i \neq j, \\ [5pt]
& -V^T_i \hsigma V_j \le \Delta_{ij}  \  \  \ \ \forall i \neq j, \\ [5pt]
& V^T V = I.   
\ea
\eeq 
Clearly, problem \eqnok{smooth-form} has the same form as \eqnok{nlp}. From 
Subsection \ref{aug-method}, we know that the sufficient conditions for convergence 
of our augmented Lagrangian method include: i) a feasible point is explicitly given; 
and ii) Robinson's condition \eqnok{rob-cond} holds at an accumulation point. It is 
easy to observe that any $V\in\Re^{n \times r}$ consisting of $r$ 
orthonormal eigenvectors of $\hsigma$ is a feasible point of \eqnok{smooth-form}, 
and thus the first condition is trivially satisfied. Given that the accumulation 
points are not known beforehand, it is hard to check the second condition directly. 
Instead, we may check Robinson's condition at all feasible points of \eqnok{smooth-form}. 
However, due to complication of the constraints, we are only able to verify Robinson's 
condition at a set of feasible points below. Before proceeding, we establish a technical 
lemma as follows that will be used subsequently. 

\begin{lemma} \label{grad-sys}
Let $V\in\Re^{n \times r}$ be a feasible solution of \eqnok{smooth-form}. Given 
any $W_1$, $W_2 \in \cS^r$, the system of
\beqa
\dV^T\hsigma V + V^T \hsigma \ \dV +\dD &=& W_1, \label{constr-1} \\
\dV^T V + V^T \dV &=& W_2 \label{constr-2}
\eeqa
has at least one solution $(\dV,\dD) \in \Re^{n\times r} \times \cD^r$ if one of 
the following conditions holds:
\bi
\item[a)] $V^T\hsigma V$ is diagonal and $V^T_i \hsigma V_i  \neq V^T_j\hsigma V_j$ 
for all $i\neq j$;   
\item[b)] $V^T\hsigma(I-VV^T)\hsigma V$ is nonsingular. 
\ei
\end{lemma}

\begin{proof}
Note that the columns of $V$ consist of $r$ orthonormal eigenvectors. Therefore, 
there exist $\bV \in\Re^{n\times (n-r)}$ such that $[V \ \bV]\in\Re^{n \times n}$ 
is an orthogonal matrix. It follows that for any $\dV \in \Re^{n \times r}$, 
there exists $\dP\in \Re^{r \times r}$ and $\dbP \in\Re^{(n-r) \times r}$ such that 
$\dV = V\dP+\bV \dbP$. Performing such a change of variable for $\dV$, 
and using the fact that the matrix $[V \ \bV]$ is orthogonal, we can show that 
the system of \eqnok{constr-1} and \eqnok{constr-2} is equivalent to
\beqa
\dP^T G  + G \dP + \dbP^T \bG + \bG^T \dbP + \dD &=& W_1, 
\label{constr-11} \\
\dP^T + \dP &=& W_2 \label{constr-21},
\eeqa
where $G=V^T\hsigma V$ and $\bG=\bV^T\hsigma V$. The remaining proof of 
this lemma reduces to show that the system of \eqnok{constr-11} and 
\eqnok{constr-21} has at least a solution $(\dP,\dbP,\dD)\in \Re^{r \times r} 
\times \Re^{(n-r) \times r} \times \cD^r$  if one of conditions (a) or 
(b) holds. 

First, we assume that condition (a) holds. Then, $G$ is a diagonal matrix 
and $G_{ii} \neq G_{jj}$ for all $i \neq j$. It follows that there exists a 
unique $\dP^*\in\Re^{n \times r}$ satisfying $\dP_{ii}=(W_2)_{ii}/2$ for 
all $i$ and 
\[
\ba{rcllcl}
\dP_{ij} G_{jj} &+& G_{ii} \dP_{ij} &=& (W_1)_{ij} \ \ \ \ 
\forall i \neq j, \\
\dP_{ij} &+& \dP_{ji} &=& (W_2)_{ij} \ \ \ \ 
\forall i \neq j.   
\ea
\] 
Now, let $\dbP^*=0$ and $\dD^* = \widetilde \Diag(W_1-GW_2)$. 
It is easy to verify that $(\dP^*,\dbP^*,\dD^*)$ is a solution of 
the system of \eqnok{constr-11} and \eqnok{constr-21}. 

We next assume that condition (b) holds. Given any $\dbP\in\Re^{(n-r) 
\times r}$, there exist $\dY\in\Re^{(n-r) \times r}$ and $\dZ\in \Re^{r \times r}$ 
such that $\bG^T\dY=0$ and $\dbP=\dY+\bG \dZ$. Peforming such a change of variable 
for $\dbP$, we see that \eqnok{constr-11} can be rewritten as 
\beq \label{constr-12}
\dP^T G  + G \dP + \dZ^T \bG^T\bG + \bG^T\bG \dZ + \dD = W_1. 
\eeq
Thus, it suffices to show that the system of \eqnok{constr-21} and \eqnok{constr-12} 
 has at least a solution $(\dP,\dZ,\dD)\in \Re^{r \times r} \times \Re^{r \times r} 
\times \cD^r$. Using the definition of $\bG$ and the fact that the matrix $[V \ \bV]$ 
is orthogonal, we see that 
\[
\bG^T \bG = V^T\hsigma \bV\bV^T\hsigma V  = V^T\hsigma(I-VV^T)\hsigma V,
\] 
which together with condition (b) implies that $\bG^T\bG$ is nonsingular. Now, let 
\[
\dP^*=W_2/2,  \ \ \ \dZ^*=(\bG^T\bG)^{-1}(2W_1-W_2G-GW_2)/4, 
\ \ \ \dD^*=0.  
\]
It is easy to verify that $(\dP^*,\dZ^*,\dD^*)$ is a solution of the system of 
\eqnok{constr-12} and \eqnok{constr-21}. Therefore, the conclusion holds.
\end{proof}   

\gap

We are now ready to show that Robinson's condition \eqnok{rob-cond} holds at a 
set of feasible points of \eqnok{smooth-form}. 

\begin{proposition} \label{constr-cond} 
Let $V\in\Re^{n \times r}$ be a feasible solution of \eqnok{smooth-form}. The 
Robinson's condition \eqnok{rob-cond} holds at $V$ if one of the following 
conditions hold:
\bi
\item[a)] $\Delta_{ij}=0$ and $V^T_i \hsigma V_i  \neq V^T_j\hsigma V_j$ for 
all $i\neq j$;
\item[b)] There is at least one active and one inactive inequality constraint 
of \eqnok{smooth-form} at $V$, and $V^T\hsigma(I-VV^T)\hsigma V$ is nonsingular;
\item[c)] All inequality constraints of \eqnok{smooth-form} are inactive at $V$.
\ei
\end{proposition}

\begin{proof}
We first suppose that condition (a) holds. Then, it immediately implies that 
$V^T\hsigma V$ is diagonal, and hence the condition (a) of Lemma \ref{grad-sys} 
holds.  In addition, we observe that all constraints of \eqnok{smooth-form} 
become equality ones. Using these facts and Lemma \ref{grad-sys}, 
we see that Robinson's condition \eqnok{rob-cond} holds at $V$. Next, we 
assume that condition (b) holds. It implies that condition (b) of 
Lemma \ref{grad-sys} holds. The conclusion then follows directly from Lemma 
\ref{grad-sys}. Finally, suppose condition (c) holds. Then, Robinson's 
condition \eqnok{rob-cond} holds at $V$ if and only if \eqnok{constr-2} has at least 
a solution $\dV\in\Re^{n\times r}$ for any $W_2\in\cS^r$. Noting that $V^TV=I$, we 
easily see that $\dV=VW_2/2$ is a solution of \eqnok{constr-2}, and thus Robinson's 
condition \eqnok{rob-cond} holds at $V$. 
\end{proof}

\gap

 From Proposition \ref{constr-cond}, we see that Robinson's condition 
\eqnok{rob-cond} indeed holds at a set of feasible points of \eqnok{smooth-form}. 
Though we are not able to show that it holds at all feasible points of 
\eqnok{smooth-form}, we observe in our implementation that the accumulation 
points of our augmented Lagrangian method generally satisfy one of the 
conditions described in Proposition \ref{constr-cond}, and so Robinson's 
condition usually holds at the accumulation points. Moreover, we have never 
seen that our augmented Lagrangian method failed to converge for an instance 
in our implementation so far. 

\subsection{Implementation details of augmented Lagrangian method for 
\eqnok{smooth-form}}
\label{implement}

In this section, we show how our augmented Lagrangian method proposed 
in Subsection \ref{aug-method} can be applied to solve problem 
\eqnok{smooth-form} (or, equivalently, \eqnok{diag-approx}). In 
particular, we will discuss the implementation details of outer 
and inner iterations of this method.

We first discuss how to efficiently evaluate the function and gradient 
involved in our augmented Lagrangian method for problem \eqnok{smooth-form}. 
Suppose that $\vrho >0$ is a penalty parameter, and $\{\lambda^+_{ij}\}_{i\neq j}$ 
and $\{\lambda^-_{ij}\}_{i\neq j}$ are the Lagrangian multipliers for the 
inequality constraints of \eqnok{smooth-form}, respectively, and 
$\mu \in \cS^r$ is the Lagrangian multipliers for the equality constraints 
of \eqnok{smooth-form}. For convenience of presentation, let $\Delta \in \cS^r$ 
be the matrix whose $ij$th entry equals the parameter $\Delta_{ij}$ of \eqnok{smooth-form} 
for all $i\neq j$ and diagonal entries are $0$. Similarly, let $\lambda^+$ (resp., 
$\lambda^-$) be an $r \times r$ symmetric matrix whose $ij$th entry is $\lambda^+_{ij}$ 
(resp., $\lambda^-_{ij}$) for all $i\neq j$ and diagonal entries are $0$. We now
define $\lambda\in\Re^{2r \times r}$ by stacking $\lambda^+$ over $\lambda^-$.
Using these notations, we observe that the associated Lagrangian function for 
problem \eqnok{smooth-form} can be rewritten as
\beq \label{augfun-spca}
L_{\vrho}(V,\lambda,\mu) =  \w(V) + \rho \bt |V|,
\eeq
where 
\beq \nn
\w(V) = - \tr(V^T \hsigma V) + \frac{1}{2\vrho}\left(\left\|\left[
\left(\ba{ll}
\lambda^+\\ 
\lambda^-
\ea\right)
+\vrho 
 \left(\ba{ll}
S - \Delta  \\
-S - \Delta 
\ea\right)\right]^+\right\|_F^2-
\left\|\left(\ba{ll}
\lambda^+ \\ 
\lambda^-
\ea\right)\right\|_F^2\right) + 
\mu \bt R + \frac{\vrho}{2} \|R\|_F^2,
\eeq
and
\beq \label{SR}
S =  V^T \hsigma V - {\widetilde\Diag}(V^T \hsigma V) , \gap R= {V}^TV-I.
\eeq
It is not hard to verify that the gradient of $w(V)$ can be computed 
according to
\[
\nabla w(V) = 2\left(-\hsigma V\left(I- [\lambda^++\vrho S- \vrho\Delta]^+
+ [\lambda^- -\vrho S- \vrho\Delta]^+\right) + V(\mu+\vrho R) \right).
\]
Clearly, the main effort for the above function and gradient evaluations 
lies in computing $V^T \hsigma V$ and $\hsigma V$. When $\hsigma\in\cS^p$ 
is explicitly given, the computational complexity for evaluating these 
two quantities is $O(p^2r)$. In practice, we are, however, typically 
given the data matrix $X\in\Re^{n\times p}$. Assuming the column means 
of $X$ are $0$, the sample covariance matrix $\hsigma$ can be obtained 
from $\hsigma=X^TX/(n-1)$. Nevertheless, when $p \gg n$, we observe that 
it is not efficient to compute and store $\hsigma$. Also, it is much 
cheaper to compute $V^T \hsigma V$ and $\hsigma V$ by using $\hsigma$ 
implicitly rather explicitly. Indeed, we can first evaluate $XV$, and 
then compute $V^T \hsigma V$ and $\hsigma V$ according to
\[
V^T \hsigma V = (XV)^T (XV)/(n-1), \gap \hsigma V = X^T(XV)/(n-1).
\]
Then, the resulting overall computational complexity is $O(npr)$, which 
is clearly much superior to the one by using $\hsigma$ explicitly, that 
is, $O(p^2r)$. 
  
We now address initialization and termination criterion for our 
augmented Lagrangian method. In particular, we choose initial point 
$V^0_\init$ and feasible point $V^{\rm feas}$ to be the loading 
vectors of the $r$ standard PCs, that is, the orthonormal eigenvectors 
corresponding to $r$ largest eigenvalues of $\hsigma$. In addition, we 
set initial penalty parameter and Lagrangian multipliers to be $1$, and 
set the parameters $\tau=0.2$ and $\sigma=10$. We terminate our method 
once the constraint violation and the relative difference between the 
augmented Lagrangian function and the regular objective function are 
sufficiently small, that is,   
\beq \label{term-aug}
\max_{i\neq j} [|V^T_i\hsigma V_j|-\Delta_{ij}]^+ \le  \eps_I, \ \ \ 
\max_{i,j} |R_{ij}| \le \eps_E, \ \ \ 
\frac{|L_{\vrho}(V,\lambda,\mu) - f(V)|} {\max {(|f(V)|,1)}} \le \eps_O,
\eeq
where $f(V)=-\tr(V^T\hsigma V) + \rho \bt |V|$, $R$ 
is defined in \eqnok{SR}, and $\eps_I$, $\eps_E$, $\eps_O$ are some 
prescribed accuracy parameters corresponding to inequality constraints, 
equality constraints and objective function, respectively. 
%In our numerical implementation, we set $\eps_I$ and $\eps_O$ to be $0.1$. However, for 
%$\eps_E$, we observe that one can choose different values for it to 
%control the orthogonality of the sparse PCs' loading vectors, therefore the value 
%of $\eps_e$ is chosen based on the solution requirement of the sparse PCA. 
 
We next discuss how to apply the nonmonotone gradient methods proposed in 
Section \ref{grad-method} for the augmented Lagrangian subproblems, which 
are in the form of 
\beq \label{aug-subprob} 
\min_V L_{\vrho}(V,\lambda,\mu),
\eeq
where the function $L_{\vrho}(\cdot,\lambda,\mu)$ is defined in \eqnok{augfun-spca}. 
Given that the implementation details of those nonmonotone gradient methods are 
similar, we only focus on the second one, that is, the nonmonotone gradient method 
II. First, the initial point for this method can be chosen according to the scheme 
described at the end of Subsection \ref{aug-method}. In addition, given an iterate 
$V^k$, the search direction $d^k$ is computed by solving subproblem \eqnok{dir} 
with $H=\alpha^{-1}_k I$, which becomes, in the context of 
\eqnok{inner-prob} and \eqnok{augfun-spca},     
\beq \label{dir-spca}
d^k := \arg\min_d \left\{\nabla w(V^k) \bt d + \frac{1}{2\alpha_k} 
 \|d\|^2_F + \rho \bt |V^k+d| \right\}.
\eeq 
Here, $\alpha_k>0$ is chosen according to the scheme proposed by Barzilai and 
Borwein \cite{BarBor88}, which is also used by Birgin et al.\ \cite{BiMaRa00} 
for studying a class of projected gradient methods. Let $0 < \alpha_{\min} < 
\alpha_{\max}$ be given. Initially, choose an arbitrary $\alpha_0 \in [\alpha_{\min},
\alpha_{\max}]$. Then, $\alpha_k$ is updated as follows:   
\[
\alpha_{k+1} = 
\left\{ \ba{ll} 
\alpha_{\max}, & \ \mbox{if} \ b_k \le 0; \\
\max \{\alpha_{\min},\min \{\alpha_{\max},a_k/b_k\}\}, &\ \mbox{otherwise},
\ea \right.
\]
where $a_k =\|V^k-V^{k-1}\|_F^2$ and $b_k= (V^k-V^{k-1}) \bt (\nabla w(V^k)-\nabla 
w(V^{k-1}))$. It is not hard to verify that the optimal solution of problem 
\eqnok{dir-spca} has a closed-form expression, which is given by
\[
d^k =\sign (C) \odot 
\left[|C|-\alpha^k \rho
\right]^+ - V^k,
\]
where $C=V^k-\alpha^k \nabla w(V^k)$. In addition, we see from Lemma 
\ref{opt-char} that the following termination criterion is suitable 
for this method when applied to \eqnok{aug-subprob}:
\[
\frac{\max_{ij}|d_I(V^k)|_{ij}}{\max(|L_{\vrho}(V^k,\lambda,\mu)|, 1)} \le \eps,
\]
where $d_I(V^k)$ is the solution of \eqnok{dir-spca} with $\alpha_k=1$, and $\eps$ 
is a prescribed accuracy parameter. In our numerical implementation, we set 
$\alpha^0 = 1/\max_{ij}|d_I(V^0)|_{ij}$, $\alpha_{\max} = 1$, $\alpha_{\min} 
= 10^{-15}$ and $\eps = 10^{-4}$.

Finally, it shall be mentioned that for the sake of practical performance, 
the numerical implementation of our augmented Lagrangian method is slightly 
different from the one described in Subsection \ref{aug-method}. In 
particular, we follow a similar scheme as discussed on pp.\ $405$ of 
\cite{Bert99} to adjust penalty parameter and Lagrangian multipliers. 
Indeed, they are updated separately rather than simultaneously. Roughly 
speaking, given $\gamma \in (0,1)$, we adjust penalty parameter only when 
the constraint violation is not decreased by a factor $\gamma$ over 
the previous minimization. Similarly, we update Lagrangian multipliers 
only when the constraint violation is decreased by a factor $\gamma$ over 
the previous minimization. We choose $\gamma=0.25$ in our implementation 
as recommended in \cite{Bert99}.           

\section{Numerical results} \label{comp} 

In this section, we conduct numerical experiments for the augmented Lagrangian 
method detailed in Subsections \ref{aug-method} and \ref{implement} for 
formulation \eqnok{smooth-form} (or, equivalently, \eqnok{diag-approx}) of 
sparse PCA on synthetic, random, and real data. In particular, we compare 
the results of our approach with several existing sparse PCA methods in 
terms of total explained variance, correlation of PCs, and orthogonality of 
loading vectors, which include the generalized power methods (Journ\'ee et 
al. \cite{JoNeRiSe08}), the DSPCA algorithm (d'Aspremont et al. 
\cite{DaElJoLa07}), the SPCA algorithm (Zou et al. \cite{ZoHaTi06}), and 
the sPCA-rSVD algorithm (Shen and Huang \cite{ShHu07}). We now list all the 
methods used in this section in Table \ref{methods}. Specifically, the methods 
with the prefix `GPower' are the generalized power methods studied in \cite{JoNeRiSe08}, 
and the method ALSPCA is the augmented Lagrangian method proposed in this 
paper.   

\begin{table}[t]
\caption{\footnotesize Sparse PCA methods used for our comparison}
\centering
\label{methods}
\begin{scriptsize}
\begin{tabular}{l l}
\hline 
GPower$_{l_1}$       & Single-unit sparse PCA via ${l_1}$-penalty  \\
GPower$_{l_0}$       & Single-unit sparse PCA via ${l_0}$-penalty  \\
GPower$_{l_{1,m}}$   & Block sparse PCA via ${l_1}$-penalty        \\
GPower$_{l_{0,m}}$   & Block sparse PCA via ${l_0}$-penalty        \\
DSPCA                & DSPCA algorithm                             \\
SPCA                 & SPCA algorithm                              \\
rSVD                 & sPCA-rSVD algorithm with soft thresholding  \\
ALSPCA               & Augmented Lagrangian algorithm              \\
\hline
\end{tabular}
\end{scriptsize}
\end{table}

As discussed in Section \ref{formulation}, the PCs obtained from the standard 
PCA based on sample covariance matrix $\hsigma \in \Re^{n\times p}$ are nearly 
uncorrelated when the sample size is sufficiently large, and the total 
explained variance by the first $r$ PCs approximately equals the sum of the 
individual variances of PCs, that is, $\tr(V^T \hsigma V)$, where $V\in\Re^{p \times r}$ 
consists of the loading vectors of these PCs. However, the PCs found by sparse 
PCA methods may be correlated with each other, and thus the quantity $\tr(V^T \hsigma V)$ 
can overestimate much the total explained variance by the PCs due to the 
overlap among their individual variances. In attempt to deal with such an  
overlap, two adjusted total explained variances were proposed in \cite{ZoHaTi06,ShHu07}. 
It is not hard to observe that they can be viewed as the total explained variance of 
a set of transformed variables from the estimated sparse PCs. Given that these 
transformed variables can be dramatically different from those sparse PCs, their 
total explained variances may differ much from each other as well. To alleviate this 
drawback while taking into account the possible correlations among PCs, we introduce 
the following {\it adjusted total explained variance} for sparse PCs:
\[
 \AV V = \tr(V^T \hsigma V) - \sqrt{\sum_{i \neq j} (V^T_i \hsigma V_j)^2}.
\]  
Clearly, when the PCs are uncorrelated, it becomes the usual total explained variance, 
that is, $\tr(V^T \hsigma V)$.  We also define the {\it cumulative percentage 
of adjusted variance} (CPAV) for the first $r$ sparse PCs as the quotient of 
the adjusted total explained variance of these PCs and the total 
explained variance by all standard PCs, that is, $\AV V/\tr(\hsigma)$.

Finally, we shall stress that the sole purpose of this section is to compare the 
performance of those methods listed in Table \ref{methods} for finding the sparse PCs 
that nearly enjoy the three important properties possessed by the standard PCA 
(see Section \ref{introduction}). Therefore, we will not compare the speed of 
these methods. Nevertheless, it shall be mentioned that our method, that is, ALSPCA, 
is a first-order method and capable of solving large-scale problems within a reasonable 
amount of time as observed in our experiments.          

\subsection{Synthetic data} \label{Zou}

In this subsection we use the synthetic data introduced by Zou et al. \cite{ZoHaTi06} 
to test the effectiveness of our approach ALSPCA for finding sparse PCs.

The synthetic example \cite{ZoHaTi06} considers three hidden factors:
\[
V_1\sim N(0,290), \gap V_2 \sim N(0,300), \gap V_3 = -0.3V_1+0.925V_2+\eps, \gap \eps\sim N(0,1),
\]
where $V_1$, $V_2$ and $\eps$ are independent. Then the 10 observable 
variables are generated as follows:
\[
 X_i = V_1+\eps_i^1, \gap \eps_i^1\sim N(0,1), \gap i=1,2,3,4,
\]
\[
 X_i = V_2+\eps_i^2, \gap \eps_i^2\sim N(0,1), \gap i=5,6,7,8,
\]
\[
 X_i = V_3+\eps_i^3, \gap \eps_i^3\sim N(0,1), \gap i=9,10,
\]
where ${\eps_i^j}$ are independent for $j=1, 2, 3$ and $i=1, \ldots,10$. We 
will use the actual covariance matrix of $(X_1,\ldots, X_{10})$ to find the 
standard and sparse PCs, respectively.

We first see that $V_1$ and $V_2$ are independent, but $V_3$ is a linear 
combination of $V_1$ and $V_2$. Moreover, the variances of the three underlying 
factors $V_1$, $V_2$ and $V_3$ are $290$, $300$, and $283.8$, respectively. 
Therefore $V_2$ is slightly more important than $V_1$, and they both are much 
more important than $V_3$. In addition, the first two standard PCs together 
explain $99.72\%$ of the total variance (see Table \ref{synthetic}). These 
facts suggest that the first two sparse PCs be sufficient to explain most of the variance. 
Ideally, the first sparse PC recovers the factor $V_2$ only using $(X_5,X_6,X_7,X_8)$, 
and the second sparse PC recovers the factor $V_1$ only using $(X_1,X_2,X_3,X_4)$. 
Moreover, given that $(X_5,X_6,X_7,X_8)$ and $(X_1,X_2,X_3,X_4)$ are independent, 
these sparse PCs would be uncorrelated and orthogonal each other.  

In our test, we set $r=2$, $\Delta_{ij}=0$ for all $i \neq j$, and $\rho=4$ for 
formulation \eqnok{smooth-form} of sparse PCA. In addition, we choose 
\eqnok{term-aug} as the termination criterion for ALSPCA with $\eps_I=\eps_O=0.1$ and 
$\eps_E=10^{-3}$. The results of standard PCA and ALSPCA for this example are presented 
in Table \ref{synthetic}. The loadings of standard and sparse PCs are given in columns 
two and three, respectively, and their CPAVs are given in the last row. We clearly 
see that our sparse PCs are consistent with the ones predicted above. Interestingly, 
they are identical with the ones obtained by SPCA and DSPCA reported in \cite{ZoHaTi06,
DaElJoLa07}. For general data, however, these methods may perform quite differently 
(see Subsection \ref{Pitp}). 

\begin{table}[t]
\caption{\footnotesize Loadings of the first two PCs by standard PCA and ALSPCA }
\centering
\label{synthetic}
\begin{scriptsize}
\begin{tabular}{|c||rr | rr|}
\hline 
\multicolumn{1}{|c||}{Variable} & \multicolumn{2}{c|}{PCA} &  
\multicolumn{2}{c|}{ALSPCA} \\
\multicolumn{1}{|c||}{} & \multicolumn{1}{c}{\sc PC1} 
& \multicolumn{1}{c|}{\sc PC2} & \multicolumn{1}{c}{\sc PC1} 
& \multicolumn{1}{c|}{\sc PC2} \\
\hline
$X_1$    &   0.1158 &  0.4785 &       0 &  0.5000  \\
$X_2$    &   0.1158 &  0.4785 &       0 &  0.5000  \\
$X_3$    &   0.1158 &  0.4785 &       0 &  0.5000  \\
$X_4$    &   0.1158 &  0.4785 &       0 &  0.5000  \\
$X_5$    &  -0.3955 &  0.1449 & -0.5000 &       0  \\
$X_6$    &  -0.3955 &  0.1449 & -0.5000 &       0  \\
$X_7$    &  -0.3955 &  0.1449 & -0.5000 &       0  \\
$X_8$    &  -0.3955 &  0.1449 & -0.5000 &       0  \\
$X_9$    &  -0.4005 & -0.0095 &       0 &       0  \\
$X_{10}$ &  -0.4005 & -0.0095 &       0 &       0  \\
\hline
\multicolumn{1}{|c||}{CPAV (\%)} & \multicolumn{2}{c|}{99.72} &  
\multicolumn{2}{c|}{80.46} \\
\hline
\end{tabular}
\end{scriptsize}
\\
\vspace{.10cm}
\text{\scriptsize Synthetic data}
\end{table}

\subsection{Random data} \label{rand}

In this subsection, we compare the performance of the GPower methods \cite{JoNeRiSe08} 
and our ALSPCA method for finding sparse PCs on a set of randomly generated data. First,
we randomly generate $100$ centered data matrices $X$ with the size of $20 \times 20$. 
Throughout this subsection, we set $\Delta_{ij} = 0.5$ for all $i \neq j$ for formulation 
\eqnok{smooth-form} of sparse PCA, and choose \eqnok{term-aug} as the termination criterion 
for ALSPCA with $\eps_I=0.1$, $\eps_E=0.1$ and $\eps_O = 0.1$. 

In the first test, we aim to find the first {\it three} sparse PCs with the average 
number of zero loadings around $30$ ($50\%$ sparsity). To meet this purpose, the 
tunning parameter $\rho$ for problem \eqnok{smooth-form} and the parameters for the 
GPower methods are properly chosen. The results of the GPower methods and ALSPCA for 
the above randomly generated instances are presented in Table \ref{rand-1}. The name 
of each method is given in column one. The sparsity measured by the number of zero 
loadings averaged over all instances is given in column two. The column three gives 
the average amount of non-orthogonality of the loading vectors, which is measured 
by the maximum absolute difference between $90^\circ$ and the angles formed by all 
pairs of loading vectors. Clearly, the smaller value in this column implies the 
better orthogonality. In addition, the maximum correlation and CPAV of sparse PCs 
averaged over all instances are presented in columns four and five, respectively. 
From Table \ref{rand-1}, we see that the average number of zero loadings 
of the first three sparse PCs for all methods are almost same, which are 
around $30$. We also observe that the sparse PCs given by our method ALSPCA 
are almost uncorrelated and their loading vectors are nearly orthogonal, 
which are much superior to the GPower methods. Moreover, our sparse PCs 
have better CPAV than the others.    

\begin{table}[t]
\caption{\footnotesize Comparison of GPower and ALSPCA}
\centering
\label{rand-1}
\begin{scriptsize}
\begin{tabular}{|l||c||c||c||c|}
\hline 
\multicolumn{1}{|c||}{Method} & \multicolumn{1}{c||}{Sparsity} &  
\multicolumn{1}{c||}{Non-orthogonality} & \multicolumn{1}{c||}{Correlation} & 
\multicolumn{1}{c|}{CPAV (\%)} \\
 
\hline
GPower$_{l_1}$     & 30.73 & 3.480 & 0.104 & 39.06 \\
GPower$_{l_0}$     & 30.37 & 3.540 & 0.111 & 38.74 \\
GPower$_{l_1,m}$   & 30.61 & 5.220 & 0.161 & 38.12 \\
GPower$_{l_0,m}$   & 30.51 & 5.216 & 0.153 & 37.99 \\
ALSPCA             & 30.91 & 1.007 & 0.038 & 40.79 \\
\hline
\end{tabular}
\end{scriptsize}
\\
\vspace{.10cm}
\text{\scriptsize Random data: Test I}
\end{table}
 
Our aim of the second test is to find the first {\it six} sparse PCs 
with the average number of zero loadings around $60$ ($50\%$ sparsity). 
To reach this goal, the tunning parameter $\rho$ for problem \eqnok{smooth-form} 
and the parameters for the GPower methods are appropriately chosen. The results of 
the GPower methods and ALSPCA for the above randomly generated instances are 
presented in Table \ref{rand-2}. Each column of Table \ref{rand-2} has 
the same meaning as Table \ref{rand-1}. First, we clearly see that our method 
substantially outperforms the GPower methods in terms of uncorrelation of PCs, 
orthogonality of loading vectors and CPAV. Further, in comparison with Table \ref{rand-1}, 
we observe that as the number of PCs increases, the CPAV grows accordingly 
for all methods, and moreover, the sparse PCs given by the GPower methods 
become much more correlated and non-orthogonal each other. But the performance 
of our sparse PCs almost remains same. This phenomenon is actually not surprising, 
given that uncorrelation and orthogonality are not well taken into account in 
the GPower methods.  
         
\begin{table}[t]
\caption{\footnotesize Comparison of GPower and ALSPCA}
\centering
\label{rand-2}
\begin{scriptsize}
\begin{tabular}{|l||c||c||c||c|}
\hline 
\multicolumn{1}{|c||}{Method} & \multicolumn{1}{c||}{Sparsity} &  
\multicolumn{1}{c||}{Non-orthogonality} & \multicolumn{1}{c||}{Correlation} & 
\multicolumn{1}{c|}{CPAV (\%)} \\
\hline
GPower$_{l_1}$     & 60.58 &  5.766 & 0.168 & 63.46 \\
GPower$_{l_0}$     & 60.53 &  6.031 & 0.169 & 63.27 \\
GPower$_{l_1,m}$   & 60.02 &  8.665 & 0.265 & 62.34 \\
GPower$_{l_0,m}$   & 60.14 &  9.146 & 0.272 & 62.00 \\
ALSPCA             & 60.74 &  1.177 & 0.066 & 64.62 \\
\hline
\end{tabular}
\end{scriptsize}
\\
\vspace{.10cm}
\text{\scriptsize Random data: Test II}
\end{table}

\subsection{Pitprops data} \label{Pitp}

In this subsection we test the performance of our approach ALSPCA for finding sparse 
PCs on the real data, that is, Pitprops data. We also compare the results with 
several existing methods \cite{ZoHaTi06,DaElJoLa07,ShHu07,JoNeRiSe08}.
%such as SPCA, DSPCA, rSVD, and ${\rm GPower}_{l_1}$.      

The Pitprops data introduced by Jeffers \cite{Jef67} has $180$ observations and $13$ measured 
variables. It is a classic example that illustrates the difficulty of interpreting PCs. Recently, 
several sparse PCA methods \cite{JoTrUd03,ZoHaTi06,ShHu07,DaElJoLa07} have been applied to this 
data set for finding {\it six} sparse PCs by using the actual covariance matrix. For ease of comparison, 
we present the standard PCs, and some of those sparse PCs in Tables \ref{Pit-PCA}-\ref{Pit-DSPCA}, 
respectively. It shall be mentioned that two groups of sparse PCs were found by DSPCA with the 
parameter $k_1=5$ or $6$, and they have similar sparsity and total explained variance (see 
\cite{DaElJoLa07} for details). Thus, we only present the latter one (i.e., the one with $k_1=6$) 
in Table \ref{Pit-DSPCA}. Also, we applied the GPower methods \cite{JoNeRiSe08} to this data for 
finding the PCs with the largest sparsity of the ones given in \cite{ZoHaTi06, 
ShHu07,DaElJoLa07}, and found the best result was given by ${\rm GPower}_{l_1}$. 
Thus, we only report the six sparse PCs obtained by ${\rm GPower}_{l_1}$ in Table \ref{Pit-GPower}. 
In addition, we present sparsity, CPAV, non-orthogonality and correlation of the PCs 
obtained by the standard PCA and sparse PCA methods \cite{ZoHaTi06,ShHu07,DaElJoLa07,JoNeRiSe08} 
in columns two to five of Table \ref{result}, respectively. In more details, the second and fifth 
columns respectively give sparsity (measured by the number of zero loadings) and CPAV. The third column
reports non-orthogonality, which is measured by the maximum absolute difference 
between $90^\circ$ and the angles formed by all pairs of loading vectors. The fourth column presents 
the maximum correlation of PCs. Though the PCs obtained by these sparse PCA methods have nice sparsity, we 
observe from Tables \ref{result} that they are much correlated, and moreover almost all 
of them are far from orthogonal each other except the ones given by SPCA \cite{ZoHaTi06}. 
To improve the quality of sparse PCs, we next apply our approach ALSPCA, and compare the 
results with these methods. For all tests below, we choose \eqnok{term-aug} as the termination 
criterion for ALSPCA with $\eps_O=0.1$ and $\eps_I=\eps_E=10^{-3}$. 

\begin{table}[t]
\caption{\footnotesize Loadings of the first six PCs by standard PCA}
\centering
\label{Pit-PCA}
\begin{scriptsize}
\begin{tabular}{|l||r r r r r r|}
\hline 
\multicolumn{1}{|r||}{Variable} & \multicolumn{1}{r}{PC1} &  
\multicolumn{1}{r}{PC2} & \multicolumn{1}{r}{PC3}&  
\multicolumn{1}{r}{PC4} & \multicolumn{1}{r}{PC5}&  
\multicolumn{1}{r|}{PC6}     \\ 
\hline
topdiam  &  0.4038 &   0.2178  &  0.2073  &  0.0912 &   0.0826 &   0.1198 \\
length   &  0.4055 &   0.1861  &  0.2350  &  0.1027 &   0.1128 &   0.1629 \\
moist    &  0.1244 &   0.5406  & -0.1415  & -0.0784 &  -0.3498 &  -0.2759 \\
testsg   &  0.1732 &   0.4556  & -0.3524  & -0.0548 &  -0.3558 &  -0.0540 \\
ovensg   &  0.0572 &  -0.1701  & -0.4812  & -0.0491 &  -0.1761 &   0.6256 \\
ringtop  &  0.2844 &  -0.0142  & -0.4753  &  0.0635 &   0.3158 &   0.0523 \\
ringbut  &  0.3998 &  -0.1897  & -0.2531  &  0.0650 &   0.2151 &   0.0026 \\
bowmax   &  0.2936 &  -0.1892  &  0.2431  & -0.2856 &  -0.1853 &  -0.0551 \\
bowdist  &  0.3566 &   0.0171  &  0.2076  & -0.0967 &   0.1061 &   0.0342 \\
whorls   &  0.3789 &  -0.2485  &  0.1188  &  0.2050 &  -0.1564 &  -0.1731 \\
clear    & -0.0111 &   0.2053  &  0.0704  & -0.8036 &   0.3430 &   0.1753 \\
knots    & -0.1151 &   0.3432  & -0.0920  &  0.3008 &   0.6003 &  -0.1698 \\
diaknot  & -0.1125 &   0.3085  &  0.3261  &  0.3034 &  -0.0799 &   0.6263 \\
\hline
\end{tabular}
\end{scriptsize}
\\
\vspace{.10cm}
\text{\scriptsize Pitprops data}
\end{table}

\begin{table}[t]
\caption{\footnotesize Loadings of the first six PCs by SPCA}
\centering
\label{Pit-SPCA}
\begin{scriptsize}
\begin{tabular}{|l||r r r r r r|}
\hline 
\multicolumn{1}{|r||}{Variable} & \multicolumn{1}{r}{PC1} &  
\multicolumn{1}{r}{PC2} & \multicolumn{1}{r}{PC3}&  
\multicolumn{1}{r}{PC4} & \multicolumn{1}{r}{PC5}&  
\multicolumn{1}{r|}{PC6}     \\
\hline
topdiam    &  -0.477 &      0 &      0 &     0 &     0 &      0 \\
length     &  -0.476 &      0 &      0 &     0 &     0 &      0 \\
moist      &       0 &  0.785 &      0 &     0 &     0 &      0 \\
testsg     &       0 &  0.620 &      0 &     0 &     0 &      0 \\
ovensg     &   0.177 &      0 &  0.640 &     0 &     0 &      0 \\
ringtop    &       0 &      0 &  0.589 &     0 &     0 &      0 \\
ringbut    &  -0.250 &      0 &  0.492 &     0 &     0 &      0 \\
bowmax     &  -0.344 & -0.021 &      0 &     0 &     0 &      0 \\
bowdist    &  -0.416 &      0 &      0 &     0 &     0 &      0 \\
whorls     &  -0.400 &      0 &      0 &     0 &     0 &      0 \\
clear      &       0 &      0 &      0 &    -1 &     0 &      0 \\
knots      &       0 &  0.013 &      0 &     0 &    -1 &      0 \\ 
diaknot    &       0 &      0 & -0.015 &     0 &     0 &      1 \\
\hline
\end{tabular}
\end{scriptsize}
\\
\vspace{.10cm}
\text{\scriptsize Pitprops data}
\end{table}

\begin{table}[t]
\caption{\footnotesize Loadings of the first six PCs by rSVD}
\centering
\label{Pit-rSVD}
\begin{scriptsize}
\begin{tabular}{|l||r r r r r r|}
\hline 
\multicolumn{1}{|r||}{Variable} & \multicolumn{1}{r}{PC1} &  
\multicolumn{1}{r}{PC2} & \multicolumn{1}{r}{PC3}&  
\multicolumn{1}{r}{PC4} & \multicolumn{1}{r}{PC5}&  
\multicolumn{1}{r|}{PC6}     \\
\hline
topdiam    &  -0.449 &       0 &      0 & -0.114 &      0 &      0 \\
length     &  -0.460 &       0 &      0 & -0.102 &      0 &      0 \\
moist      &       0 &  -0.707 &      0 &      0 &      0 &      0 \\
testsg     &       0 &  -0.707 &      0 &      0 &      0 &      0 \\
ovensg     &       0 &       0 &  0.550 &      0 &      0 & -0.744 \\
ringtop    &  -0.199 &       0 &  0.546 & -0.176 &      0 &      0 \\
ringbut    &  -0.399 &       0 &  0.366 &      0 &      0 &      0 \\
bowmax     &  -0.279 &       0 &      0 &  0.422 &      0 &      0 \\
bowdist    &  -0.380 &       0 &      0 &  0.283 &      0 &      0 \\
whorls     &  -0.407 &       0 &      0 &      0 &  0.231 &      0 \\
clear      &       0 &       0 &      0 & -0.785 & -0.973 &      0 \\
knots      &       0 &       0 &      0 & -0.265 &      0 &  0.161 \\ 
diaknot    &       0 &       0 & -0.515 &      0 &      0 & -0.648 \\
\hline
\end{tabular}
\end{scriptsize}
\\
\vspace{.10cm}
\text{\scriptsize Pitprops data}
\end{table}

\begin{table}[t]
\caption{\footnotesize Loadings of the first six PCs by DSPCA}
\centering
\label{Pit-DSPCA}
\begin{scriptsize}
\begin{tabular}{|l||r r r r r r|}
\hline 
\multicolumn{1}{|r||}{Variable} & \multicolumn{1}{r}{PC1} &  
\multicolumn{1}{r}{PC2} & \multicolumn{1}{r}{PC3}&  
\multicolumn{1}{r}{PC4} & \multicolumn{1}{r}{PC5}&  
\multicolumn{1}{r|}{PC6}     \\
\hline
topdiam    & -0.4907 &       0 &       0 &        0 &       0 &        0 \\
length     & -0.5067 &       0 &       0 &        0 &       0 &        0 \\
moist      &       0 &  0.7071 &       0 &        0 &       0 &        0 \\
testsg     &       0 &  0.7071 &       0 &        0 &       0 &        0 \\
ovensg     &       0 &       0 &       0 &        0 & -1.0000 &        0 \\
ringtop    & -0.0670 &       0 & -0.8731 &        0 &       0 &        0 \\
ringbut    & -0.3566 &       0 & -0.4841 &        0 &       0 &        0 \\
bowmax     & -0.2335 &       0 &       0 &        0 &       0 &        0 \\
bowdist    & -0.3861 &       0 &       0 &        0 &       0 &        0 \\
whorls     & -0.4089 &       0 &       0 &        0 &       0 &        0 \\
clear      &       0 &       0 &       0 &        0 &       0 &   1.0000 \\
knots      &       0 &       0 &       0 &   1.0000 &       0 &        0 \\ 
diaknot    &       0 &       0 &  0.0569 &        0 &       0 &        0 \\
\hline
\end{tabular}
\end{scriptsize}
\\
\vspace{.10cm}
\text{\scriptsize Pitprops data}
\end{table}

In the first experiment, we aim to find six uncorrelated and orthogonal sparse PCs by ALSPCA 
while explaining most of variance. In particular, we set $r=6$, $\Delta_{ij} = 0.07$ for 
all $i \neq j$ and $\rho=0.8$ for formulation \eqnok{smooth-form} of sparse PCA. The 
resulting sparse PCs are presented in Table \ref{Pit-ALSPCA1}, and their sparsity, 
CPAV, non-orthogonality and correlation are reported in row seven of Table 
\ref{result}. We easily observe that our method ALSPCA overally outperforms the 
other sparse PCA methods substantially in all aspects except sparsity. Naturally, 
we can improve the sparsity by increasing the values of $\rho$, yet the total 
explained variance may be sacrificed as demonstrated in our next experiment.           

\begin{table}[t]
\caption{\footnotesize Loadings of the first six PCs by GPower$_{l_1}$}
\centering
\label{Pit-GPower}
\begin{scriptsize}
\begin{tabular}{|l||r r r r r r|}
\hline 
\multicolumn{1}{|r||}{Variable} & \multicolumn{1}{r}{PC1} &  
\multicolumn{1}{r}{PC2} & \multicolumn{1}{r}{PC3}&  
\multicolumn{1}{r}{PC4} & \multicolumn{1}{r}{PC5}&  
\multicolumn{1}{r|}{PC6}     \\
 
\hline
topdiam      &  -0.4182 &        0 &       0 &       0 &       0 &      0 \\
length       &  -0.4205 &        0 &       0 &       0 &       0 &      0 \\
moist        &        0 &  -0.7472 &       0 &       0 &       0 &      0 \\
testsg       &  -0.1713 &  -0.6646 &       0 &       0 &       0 &      0 \\
ovensg       &        0 &        0 &       0 &       0 & -0.7877 &      0 \\
ringtop      &  -0.2843 &        0 &       0 &       0 & -0.6160 &      0 \\
ringbut      &  -0.4039 &        0 &       0 &       0 &       0 &      0 \\
bowmax       &  -0.3002 &        0 &       0 &       0 &       0 &      0 \\
bowdist      &  -0.3677 &        0 &       0 &       0 &       0 &      0 \\
whorls       &  -0.3868 &        0 &       0 &       0 &       0 &      0 \\
clear        &        0 &        0 &       0 &       0 &       0 & 1.0000 \\
knots        &        0 &        0 &       0 &  1.0000 &       0 &      0 \\ 
diaknot      &        0 &        0 &  1.0000 &       0 &       0 &      0 \\
\hline
\end{tabular}
\end{scriptsize}
\\
\vspace{.10cm}
\text{\scriptsize Pitprops data}
\end{table}

\begin{table}[t]
\caption{\footnotesize Loadings of the first six PCs by ALSPCA }
\centering
\label{Pit-ALSPCA1}
\begin{scriptsize}
\begin{tabular}{|l||r r r r r r|}
\hline 
\multicolumn{1}{|r||}{Variable} & \multicolumn{1}{r}{PC1} &  
\multicolumn{1}{r}{PC2} & \multicolumn{1}{r}{PC3}&  
\multicolumn{1}{r}{PC4} & \multicolumn{1}{r}{PC5}&  
\multicolumn{1}{r|}{PC6}     \\
\hline
topdiam  &   0.4394 &       0 &        0 &       0 &       0 &        0 \\
length   &   0.4617 &       0 &        0 &       0 &       0 &        0 \\
moist    &   0.0419 &  0.4611 &  -0.1644 &  0.0688 & -0.3127 &        0 \\
testsg   &   0.1058 &  0.7902 &        0 &       0 &       0 &        0 \\
ovensg   &   0.0058 &       0 &        0 &       0 &       0 &        0 \\
ringtop  &   0.1302 &       0 &   0.2094 &       0 &       0 &   0.9999 \\
ringbut  &   0.3477 &       0 &   0.0515 &       0 &  0.3240 &        0 \\
bowmax   &   0.2256 & -0.3566 &        0 &       0 &       0 &        0 \\
bowdist  &   0.4063 &       0 &        0 &       0 &       0 &        0 \\
whorls   &   0.4606 &       0 &        0 &       0 &       0 &  -0.0125 \\
clear    &        0 &  0.0369 &        0 & -0.9973 &       0 &        0 \\
knots    &  -0.1115 &  0.1614 &  -0.0762 &  0.0239 &  0.8929 &        0 \\ 
diaknot  &  -0.0487 &  0.0918 &   0.9595 &  0.0137 &       0 &        0 \\
\hline
\end{tabular}
\end{scriptsize}
\\
\vspace{.10cm}
\text{\scriptsize Pitprops data: Test I}
\end{table}

We now attempt to find six PCs with similar correlation and orthogonality but higher 
sparsity than those given in the above experiment. For this purpose, we set $\Delta_{ij} = 0.07$ 
for all $i \neq j$ and choose $\rho=2.1$ for problem \eqnok{smooth-form} in this experiment. 
The resulting sparse PCs are presented in Table \ref{Pit-ALSPCA2}. The CPAV, non-orthogonality 
and correlation of these PCs are given in row eight of Table \ref{result}. 
In comparison with the ones given in the above experiment, the PCs obtained in this experiment 
are much more sparse while retaining almost same correlation and orthogonality. However, 
their CPAV goes down dramatically. Combining the results of these two experiments, we 
deduce that for the Pitprops data, it seems not possible to extract six highly sparse 
(e.g., around $60$ zero loadings), nearly orthogonal and uncorrelated PCs while explaining 
most of variance as they may not exist. The following experiment further sustains such 
a deduction. 

\begin{table}[t]
\caption{\footnotesize Loadings of the first six PCs by ALSPCA}
\centering
\label{Pit-ALSPCA2}
\begin{scriptsize}
\begin{tabular}{|l||r r r r r r|}
\hline 
\multicolumn{1}{|r||}{Variable} & \multicolumn{1}{r}{PC1} &  
\multicolumn{1}{r}{PC2} & \multicolumn{1}{r}{PC3}&  
\multicolumn{1}{r}{PC4} & \multicolumn{1}{r}{PC5}&  
\multicolumn{1}{r|}{PC6}     \\
\hline
topdiam  &   1.0000 &        0 &       0 &       0 &       0 &        0 \\
length   &        0 &  -0.2916 & -0.1421 &       0 &       0 &  -0.0599 \\
moist    &        0 &   0.9565 & -0.0433 &       0 &       0 &  -0.0183 \\
testsg   &        0 &        0 &       0 &  0.0786 & -0.1330 &        0 \\
ovensg   &        0 &        0 & -0.9683 &       0 &       0 &        0 \\
ringtop  &        0 &        0 &       0 &       0 &       0 &        0 \\
ringbut  &        0 &        0 &  0.1949 &       0 &  0.2369 &        0 \\
bowmax   &        0 &        0 &       0 &       0 &       0 &        0 \\
bowdist  &        0 &        0 &       0 &       0 &       0 &        0 \\
whorls   &        0 &        0 &       0 &       0 &       0 &        0 \\
clear    &        0 &        0 &       0 & -0.9969 &       0 &        0 \\
knots    &        0 &        0 & -0.0480 &  0.0109 &  0.9624 &        0 \\ 
diaknot  &        0 &        0 & -0.0093 &       0 &       0 &   0.9980 \\
\hline
\end{tabular}
\end{scriptsize}
\\
\vspace{.10cm}
\text{\scriptsize Pitprops data: Test II}
\end{table}

 Finally we are interested in exploring how the correlation controlling parameters 
$\Delta_{ij} (i \neq j)$ affect the performance of the sparse PCs. In particular, 
we set $\Delta_{ij} = 0.5$ for all $i \neq j$ and choose $\rho=0.7$ for problem 
\eqnok{smooth-form}. The obtained sparse PCs are presented in Table \ref{Pit-ALSPCA3}. 
The CPAV, non-orthogonality and correlation of these PCs are given 
in the last row of Table \ref{result}. We see that these PCs are highly sparse,  
orthogonal, and explain good amount of variance. However, they are quite 
correlated with each other, which is actually not surprising, given that $\Delta_{ij} 
(i \neq j)$ are not small. Despite such a drawback, we observe that these sparse PCs 
still overally outperform those obtained by SPCA, rSVD, DSPCA and ${\rm GPower}_{l_1}$.

 From the above experiments, we may conclude that for the Pitprops data, there do 
not exist six highly sparse, nearly orthogonal and uncorrelated PCs while explaining 
most of variance. Therefore, the most acceptable sparse PCs seem to be the ones given 
in Table \ref{Pit-ALSPCA1}.     

\begin{table}[t]
\caption{\footnotesize Loadings of the first six PCs by ALSPCA}
\centering
\label{Pit-ALSPCA3}
\begin{scriptsize}
\begin{tabular}{|l||r r r r r r|}
\hline 
\multicolumn{1}{|r||}{Variable} & \multicolumn{1}{r}{PC1} &  
\multicolumn{1}{r}{PC2} & \multicolumn{1}{r}{PC3}&  
\multicolumn{1}{r}{PC4} & \multicolumn{1}{r}{PC5}&  
\multicolumn{1}{r|}{PC6}     \\
\hline
topdiam  &   0.4051 &       0 &        0 &       0 &       0 &        0 \\
length   &   0.4248 &       0 &        0 &       0 &       0 &        0 \\
moist    &        0 &  0.7262 &        0 &       0 &       0 &        0 \\
testsg   &   0.0018 &  0.6875 &        0 &       0 &       0 &        0 \\
ovensg   &        0 &       0 &  -1.0000 &       0 &       0 &        0 \\
ringtop  &   0.1856 &       0 &        0 &       0 &       0 &        0 \\
ringbut  &   0.4123 &       0 &        0 &       0 &       0 &        0 \\
bowmax   &   0.3278 &       0 &        0 &       0 &       0 &        0 \\
bowdist  &   0.3830 &       0 &        0 &       0 &       0 &        0 \\
whorls   &   0.4437 & -0.0028 &        0 &       0 &       0 &        0 \\
clear    &        0 &       0 &        0 & -1.0000 &       0 &        0 \\
knots    &        0 &       0 &        0 &       0 &  1.0000 &        0 \\ 
diaknot  &        0 &       0 &        0 &       0 &       0 &   1.0000 \\
\hline
\end{tabular}
\end{scriptsize}
\\
\vspace{.10cm}
\text{\scriptsize Pitprops data: Test III}
\end{table}

\begin{table}[t]
\caption{\footnotesize Comparison of SPCA, rSVD, DSPCA, GPower$_{l_1}$ and ALSPCA}
\centering
\label{result}
\begin{scriptsize}
\begin{tabular}{|l||c||c||c||c|}
\hline 
\multicolumn{1}{|c||}{Method} &  \multicolumn{1}{c||}{Sparsity} &
\multicolumn{1}{c||}{Non-orthogonality} & \multicolumn{1}{|c||}{Correlation } & 
\multicolumn{1}{|c|}{CPAV  (\%)} \\
\hline
PCA              & 0       & 0        & 0       & 87.00   \\    
SPCA             & 60      & 0.86     & 0.395   & 66.21   \\                 
rSVD             & 53      & 14.76    & 0.459   & 67.04   \\          
DSPCA            & 63      & 13.63    & 0.573   & 60.97   \\           
GPower$_{l_1}$   & 63      & 10.09    & 0.353   & 64.15   \\
ALSPCA-1         & 46      & 0.03     & 0.082   & 69.55   \\
ALSPCA-2         & 60      & 0.03     & 0.084   & 39.42   \\
ALSPCA-3         & 63      & 0.00     & 0.222   & 65.97   \\
\hline
\end{tabular}
\end{scriptsize}
\\
\vspace{.10cm}
\text{\scriptsize Pitprops data}
\end{table}

\section{Concluding remarks}
\label{concl-remark}

In this paper we proposed a new formulation of sparse PCA for finding 
sparse and nearly uncorrelated principal components (PCs) with orthogonal 
loading vectors while explaining as much of the total variance as possible. 
We also developed a novel globally convergent augmented Lagrangian method 
for solving a class of nonsmooth constrained optimization problems, which 
is well suited for our formulation of sparse PCA. Additionally, we proposed 
two nonmonotone gradient methods for solving the augmented Lagrangian 
subproblems, and established their global and local convergence. Finally, 
we compared our sparse PCA approach with several existing methods on 
synthetic, random, and real data, respectively. The computational 
results demonstrate that the sparse PCs produced by our approach 
substantially outperform those by other methods in terms of total 
explained variance, correlation of PCs, and orthogonality of 
loading vectors.     

As observed in our experiments, formulation \eqnok{diag-approx} is very 
effective in finding the desired sparse PCs. However, there remains a 
natural theoretical question about it. Given a set of random variables, 
suppose there exist sparse and uncorrelated PCs with orthogonal loading 
vectors while explaining most of variance of the variables. In other 
words, their actual covariance matrix $\Sigma$ has few dominant eigenvalues 
and the associated orthonormal eigenvectors are sparse. Since in practice 
$\Sigma$ is typically unknown and only approximated by a sample covariance 
matrix $\hsigma$, one natural question is whether or not there exist some 
suitable parameters $\rho$ and $\Delta_{ij} \ (i\neq j)$ so that 
\eqnok{diag-approx} is able to recover those sparse PCs almost surely as 
the sample size becomes sufficiently large.     

In Section \ref{aug-spca} we showed that Robinson's condition \eqnok{rob-cond} 
holds at a set of feasible points of \eqnok{smooth-form}. Also, we observed from 
our experiments that the accumulation points of our augmented Lagrangian method 
lie in this set when applied to \eqnok{smooth-form}, and thus it converges. 
However, it remains open whether or not Robinson's condition holds at all 
feasible points of \eqnok{smooth-form}.       

In addition, Burer and Monteiro \cite{BurMon03} recently applied the classical 
augmented Lagrangian method to a nonconvex nonlinear program (NLP) reformulation 
of semidefinite programs (SDP), and they obtained some nice computational 
results especially for the SDP relaxations of several hard combinatorial 
optimization problems. However, the classical augmented Lagrangian method 
generally cannot guarantee converging to a feasible point when applied to 
a nonconvex NLP. Due this and \cite{Mont09}, at least theoretically, their 
approach \cite{BurMon03} may not converge to a feasible point of the primal 
SDP. Given that the augmented Lagrangian method proposed in this paper 
converges globally under some mild assumptions, it would be interesting 
to apply it to the NLP reformulation of SDP and compare the performance 
with the approach studied in \cite{BurMon03}.             

Finally, the codes of our approach for solving the sparse PCA formulation \eqnok{smooth-form} 
(or, equivalently, \eqnok{diag-approx}) are written in Matlab, which are available 
online at www.math.sfu.ca/$\sim$zhaosong. As a future research, we will further 
improve their performance by conducting more extensive computational experiments 
and exploring more practical applications. 
  
\section*{Acknowledgement}

We gratefully acknowledge comments from Jim Burke, Terry Rockafellar, Defeng Sun 
and Paul Tseng in the West Coast Optimization Meeting at University of Washington, 
Seattle, USA in Spring 2009.

\end{document}